\documentclass[12pt]{amsart}

\title{Parametrized Homology via Zigzag Persistence}
\author{
Gunnar Carlsson,
Vin de Silva,
Sara Kali\v{s}nik,
Dmitriy Morozov}
  \thanks{Corresponding author: Sara Kali\v{s}nik (\texttt{skalisnikver@wesleyan.edu})}
\date{}
\usepackage{fullpage}
\usepackage{amsfonts}
\usepackage{amsmath}
\usepackage{amssymb}
\usepackage{amsthm}
\usepackage{mathrsfs}
\usepackage[usenames,dvipsnames]{xcolor}
\usepackage{pdfcolmk}
\usepackage{marvosym}
\usepackage[normalem]{ulem} % for \sout
\usepackage{microtype}
\usepackage{hyperref}

\usepackage{tikz-cd}
\usepackage{tikz}
\usetikzlibrary{arrows,matrix}
\usetikzlibrary{decorations.markings}

\usepackage[T1]{fontenc}
\usepackage{lmodern}

\newcommand{\ff}{\mathbf{F}}
\newcommand{\kk}{\mathbf{k}}
\newcommand{\qq}{\mathbf{Q}}
\newcommand{\rr}{\mathbf{R}}

\newcommand{\RR}{\mathbb{R}}
\newcommand{\UU}{\mathbb{U}}
\newcommand{\VV}{\mathbb{V}}
\newcommand{\WW}{\mathbb{W}}
\newcommand{\XX}{\mathbb{X}}
\newcommand{\YY}{\mathbb{Y}}

\newcommand{\II}{\mathbb{I}}

\newcommand{\Top}{\mathsf{Top}}
\newcommand{\Int}{\mathsf{Int}}
\newcommand{\Vect}{\mathsf{Vect}}

\newcommand{\Hgr}{\mathrm{H}}
\newcommand{\Wgr}{\mathrm{W}}

\newcommand{\card}{\operatorname{card}}
\newcommand{\id}{\operatorname{id}}

\newcommand{\Img}{\operatorname{Im}}
\newcommand{\Ker}{\operatorname{Ker}}
\newcommand{\Hom}{\operatorname{Hom}}

\newcommand{\Dgm}{\operatorname{Dgm}}
\newcommand{\dgm}{\operatorname{Dgm}_{\rm u}}
\newcommand{\DgmZZ}{\Dgm^{\mathrm{ZZ}}}

\newcommand{\Rect}{\operatorname{Rect}}

\newcommand{\metric}{\mathrm{d}}
\newcommand{\bottle}{\mathrm{d_b}}
\newcommand{\cost}{\operatorname{cost}}

\numberwithin{equation}{section}

\theoremstyle{plain}
\newtheorem{theorem}[equation]{Theorem}
\newtheorem{lemma}[equation]{Lemma}
\newtheorem{proposition}[equation]{Proposition}

\newtheorem*{criterionA}{Criterion A}
\newtheorem*{criterionB}{Criterion B}

\theoremstyle{definition}
\newtheorem{definition}[equation]{Definition}
\newtheorem{example}[equation]{Example}

\theoremstyle{remark}
\newtheorem*{remark}{Remark}
\newtheorem*{example*}{Example}

 { \begin{list}%
         {$\bullet$}%
         {\setlength{\labelwidth}{20pt}%
          \setlength{\leftmargin}{25pt}%
          \setlength{\topsep}{0pt}
          \setlength{\itemsep}{1ex}
          \setlength{\parsep}{0pt}}}%
 { \end{list} }

\newcommand{\dd}{{\backslash\!\backslash}}
\newcommand{\ddt}{\protect\raisebox{0.35ex}{$\scriptstyle{\dd}$}}
\newcommand{\dds}{{^{\scriptscriptstyle\dd}}}

\newcommand{\du}{{\backslash\!\slash}}
\newcommand{\dut}{\protect\raisebox{0.35ex}{$\scriptstyle{\du}$}}
\newcommand{\dus}{{^{\scriptscriptstyle\du}}}

\newcommand{\ud}{{\slash\!\backslash}}
\newcommand{\udt}{\protect\raisebox{0.35ex}{$\scriptstyle{\ud}$}}
\newcommand{\uds}{{^{\scriptscriptstyle\ud}}}

\newcommand{\uu}{{\slash\!\slash}}
\newcommand{\uut}{\protect\raisebox{0.35ex}{$\scriptstyle{\uu}$}}
\newcommand{\uus}{{^{\scriptscriptstyle\uu}}}

\newcommand{\xx}{{\slash\!\slash\!\!\!\!\!\backslash\!\backslash}}
\newcommand{\xxt}{\protect\raisebox{0.35ex}{$\scriptstyle{\xx}$}}
\newcommand{\xxs}{{^{\scriptscriptstyle\xx}}}

\newcommand{\Hp}{\mathcal{H}}

\newcommand{\mudd}{\mu^\dds}
\newcommand{\mudu}{\mu^\dus}
\newcommand{\muud}{\mu^\uds}
\newcommand{\muuu}{\mu^\uus}

\newcommand{\nudu}{\nu^\dus}

\newcommand{\muOrd}{\mu^{\operatorname{Ord}}}
\newcommand{\muRel}{\mu^{\operatorname{Rel}}}
\newcommand{\muExtP}{\mu^{\operatorname{Ext}^+}}
\newcommand{\muExtM}{\mu^{\operatorname{Ext}^-}}

\newcommand{\ordPic}{\includegraphics[scale=1.8]{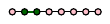}}
\newcommand{\relPic}{\includegraphics[scale=1.8]{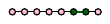}}
\newcommand{\extPPic}{\includegraphics[scale=1.8]{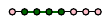}}
\newcommand{\extMPic}{\includegraphics[scale=1.8]{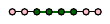}}

\newcommand{\wud}{%
	\raisebox{-0.5ex}{%
		\includegraphics[scale=1.8]{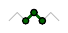}%
	}}

\newcommand{\zzboxuu}[1] 
    {\left\langle \raisebox{-1.25ex}{\includegraphics[scale=1.8,page=#1]{zz-box-uu2}}\right\rangle}
    
    \newcommand{\morse}[1] 
    {\left\langle \raisebox{-1.25ex}{\includegraphics[scale=1.8,page=#1]{morsetype.pdf}}\right\rangle}
    
    \newcommand{\zzhsplituuu}[1] 
    {\left\langle \raisebox{-1.25ex}{\includegraphics[scale=1.8,page=#1]{zz-uu-horiz-split2}}  \right\rangle}

\newcommand{\muLzzEP}[1]
    {\left\langle \raisebox{-2.5ex}{\includegraphics[scale=1.8,page=#1]{lzz-ep-uu-step-by-step}} \right\rangle}

\newcommand{\bat}[1]
    { \raisebox{-2.5ex}{\includegraphics[scale=1.8,page=#1]{batman}}}
   
   \newcommand{\critical}[1]
    { \raisebox{-0.75ex}{\includegraphics[scale=1.8,page=#1]{critical}}}
    
    \newcommand{\onof}[1]
    { \raisebox{-0.25ex}{\includegraphics[scale=1.8,page=#1]{onof}}}
    
        \newcommand{\eksample}[1]
    { \raisebox{-0.25ex}{\includegraphics[scale=1.8,page=#1]{example}}}

\newcommand{\zzspider}[1] 
    { \raisebox{-0.5ex}{\includegraphics[scale=1.8,page=#1]{spider}}}

\usepackage{graphicx}
\newcommand{\zzud}{%
 \raisebox{-0.5ex}{%
  \includegraphics[scale=1.8]{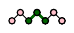}%
 }}

 \usepackage{graphicx}
\newcommand{\zzdd}{%
 \raisebox{-0.5ex}{%
  \includegraphics[scale=1.8]{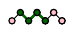}%
 }}
 
 \usepackage{graphicx}
\newcommand{\zzuu}{%
 \raisebox{-0.5ex}{%
  \includegraphics[scale=1.8]{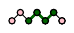}%
 }}
 
 \usepackage{graphicx}
\newcommand{\zzdu}{%
 \raisebox{-0.5ex}{%
  \includegraphics[scale=1.8]{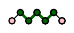}%
 }}
 
  \usepackage{graphicx}

\newcommand{\eps}{\epsilon}

\begin{document}
 \maketitle
\begin{abstract}
This paper introduces parametrized homology, a continuous-parameter generalization of levelset zigzag persistent homology that captures the behavior of the homology of the fibers of a real-valued function on a topological space. This information is encoded as a `barcode' of real intervals, each corresponding to a homological feature supported over that interval; or, equivalently, as a persistence diagram. Points in the persistence diagram are classified algebraically into four classes; geometrically, the classes identify the distinct ways in which homological features perish at the boundaries of their interval of persistence. We study the conditions under which spaces fibered over the real line have a well-defined parametrized homology; we establish the stability of these invariants; and we show how the four classes of persistence diagram correspond to the four diagrams that appear in the theory of extended persistence.
\end{abstract}

%{\bf Keywords:} Persistence Diagram, Zigzag Persistence, Levelset Zigzag Persistence.


%\newpage
%\tableofcontents

%\newpage

%-----------------------------------------------------------------
%-----------------------------------------------------------------
\section{Introduction}
\label{sec:intro}

Persistent homology is one of the key topological methods used in data analysis; as such it deserves substantial credit for the emergence of applied topology as a field. A common theme in this history has been the introduction of a method, motivated by applications or computation, that is encumbered by restrictive theoretical assumptions. The original persistent homology~\cite{pershom} required discretization of the input, an assumption that was lifted as the theory became better understood~\cite{pfd-persistence,Thestructureandstability}. The celebrated stability result~\cite{Cohen-Steiner_2007} had strong tameness assumptions that were relaxed over a sequence of papers~\cite{proximity,matchings-barcodes,Thestructureandstability}.
Viewed in this context, our paper is another rung on the climb to a transparent theory of persistence, free of unnecessary restrictions.

The specific goal of this paper is to generalize levelset zigzag persistence~\cite{Zigzagpersistenthomologyandreal} to the continuous case, lifting the restriction that the spaces under consideration have discrete structure. Our main tools are the theory of rectangular measures and a graphical notation for quiver representation calculations; both taken from~\cite{Thestructureandstability}.
On the algebraic side there are some technical requirements, regarding choice of homology theory, that we work through in detail.
On the geometric side, we study the different phenomena recorded by our invariants.
Finally, we generalize the equivalence~\cite{Zigzagpersistenthomologyandreal} between levelset persistence and extended persistence~\cite{ExtendingPersistence} to the continuous case; and we discuss parametrized cohomology.

\medskip
The general set-up is this.
Let $X$ be a topological space and let $f: X \to \RR$ be a continuous function. Such a pair $\XX = (X,f)$ is commonly called a \emph{space fibered over the real line}; in this paper, we use the convenient term \emph{$\RR$-space}. We can view an $\RR$-space as a collection of topological spaces
\[
\XX_a^a = f^{-1}(a), \qquad a \in \RR
\]
called the \emph{levelsets} of~$\XX$, where the topology on the total space~$X$ bestows upon this collection of spaces the structure of a `family'. In particular, the \emph{interlevelsets}
\[
\XX_a^b = f^{-1}[a,b], \qquad a,b \in \RR,\, a \leq b
\]
provide cobordism-style relationships between the levelsets.
The basic question is to understand the homological invariants of~$\XX$. In particular, how does the homology of~$\XX_a^a$ vary with~$a$?
Taking the family structure into account, this question demands a richer answer than simply recording the homology of each~$\XX_a^a$ separately.

What we seek is a reasonable theory for taking an $\RR$-space and decomposing its homological information into discrete features supported over intervals. To shed light on the meaning of `reasonable', we highlight some desired properties. Such a theory would:
\begin{itemize}
\item
retrieve all obvious homological information stored in $(X,f)$;
\item
be manifestly symmetric with respect to reversal of the real line~$\RR$;
\item
be widely applicable, free from excessively strong finiteness assumptions.
\end{itemize}
We return to the question of what we mean by `all obvious homological information'. First, we consider four examples of existing theories, indicating why they do not fully satisfy these properties.

\begin{example}[standard persistent homology]
The classical theory of persistence~\cite{pershom} is defined in terms of the \emph{sublevelsets}
\[
\XX^a = f^{-1}(-\infty,a]
\]
of the $\RR$-space $(X,f)$. We begin by choosing a finite set of values $a_0 < a_1 < \dots < a_n$. This could be the set of critical values in the case of a manifold with a Morse function; or it could simply be an arbitrary discretization of the real line. We then form the diagram of topological spaces
\[
\begin{tikzpicture}[xscale=1.35,yscale=1]
\draw
(0,0) node(00){$\XX^{a_0}$}
(1,0) node(10){$\XX^{a_1}$}
(2,0) node(20){$\ldots$} 
(3,0) node(30){$\XX^{a_n}$,}
;
\draw[->] (00) -- (10); 
\draw[->] (10) -- (20);
\draw[->] (20) -- (30);
\end{tikzpicture}
\]
where the arrows denote the canonical inclusion maps. By applying a homology functor~$\Hgr$ with field coefficients, we get a diagram of vector spaces and linear maps
\[
\begin{tikzpicture}[xscale=1.75,yscale=1]
\draw
    (0,0) node(00){$\Hgr(\XX^{a_0})$}
    (1.1,0) node(10){$\Hgr(\XX^{a_1})$}
    (2,0) node(20){$\ldots$}
    (3,0) node(30){$\Hgr(\XX^{a_n})$.} ;
\draw[->] (00) -- (10); 
\draw[->] (10) -- (20);
\draw[->] (20) -- (30);
\end{tikzpicture}
\]
The structure of such a diagram is described by its \emph{barcode} or \emph{persistence diagram} (Section~\ref{subsec:zigzag}).
The resulting collection of barcodes captures some of the information that we are seeking in the present work.

Standard persistent homology doesn't satisfy all our desired properties.  Although it is possible to get rid of the finite discretization of the real line~\cite{pfd-persistence,Thestructureandstability}, the first two properties are not satisfied. Most obviously, the construction is asymmetric when reversing the real line. For instance, let $X$ be the cone on a topological space~$Y$
\[
X = ( Y \times [0,1] ) / ( Y \times \{0\} )
\]
and let $f([x,t]) = t$ be the cone height function. Then the persistent homology of $(X,f)$ is indistinguishable from the persistent homology of a 1-point $\RR$-space $(*,0)$. On the other hand, the persistent homology of $(X,-f)$ detects the homology of~$Y$ over the interval $[-1,0)$.
%
%For instance, consider $f(x) = x^2$. All its sublevelsets above $a=0$ have one connected component, and the persistence diagram contains exactly one point, $[0, \infty)$. Reversing the function, the sublevelsets of $-f$ have two components, which merge together at $a=0$, and the persistence diagram of $-f$ has two points, $(\infty, 0], (\infty, -\infty)$.
%
\end{example}

One might imagine that the persistent homology of $(X,f)$ and $(X,-f)$ together capture all information of interest. The next example shows that there is, in fact, more information to be gathered.

\begin{example}[extended persistent homology]
The theory of `extended persistence' introduced by Cohen-Steiner et al.~\cite{ExtendingPersistence} has similar goals to ours, but addresses them under the restriction that $\XX$ be `tame' in the sense of having finitely many critical values and cylindrical behavior, i.e.\ `Morse-like' behavior, between those critical values. Adding \emph{superlevelsets}
\[
\XX_a = f^{-1}[a,+\infty),
\]
Cohen-Steiner et al.\ consider the sequence of spaces and pairs
\[
\begin{tikzpicture}[xscale=1.35,yscale=1]
\draw
    (0,0) node(00){$\XX^{a_0}$}
    (1,0) node(10){$\ldots$}
    (2,0) node(20){$\XX^{a_n}$}
    (3,0) node(30){$X$}
    (4.3,0) node(40){$(X, \XX_{a_n})$}
    (5.6,0) node(50){$\ldots$}
    (6.9,0) node(60){$(X, \XX_{a_0})$,}   ;

\draw[->] (00) -- (10); 
\draw[->] (10) -- (20);
\draw[->] (20) -- (30);
\draw[->] (30) -- (40);
\draw[->] (40) -- (50);
\draw[->] (50) -- (60);
\end{tikzpicture}
\]
where $a_0 < \dots < a_n$ is the set of critical values.
The extended persistence of~$\XX$ is the persistent homology of this sequence
\[
\begin{tikzpicture}[xscale=1.75,yscale=1]
\draw (0,0) node(00){$\Hgr(\XX^{a_0})$} (1,0) node(10){$\ldots$}  (2,0) node(20){$\Hgr(\XX^{a_n})$}  (3.22,0) node(30){$\Hgr(X)$} (4.6,0) node(40){$\Hgr(X, \XX_{a_n})$} (5.75,0) node(50){$\ldots$}   (6.9,0) node(60){$\Hgr(X, \XX_{a_0})$}   ;

\draw[->] (00) -- (10); 
\draw[->] (10) -- (20);
\draw[->] (20) -- (30);
\draw[->] (30) -- (40);
\draw[->] (40) -- (50);
\draw[->] (50) -- (60);
\end{tikzpicture}
\]
obtained by applying a homology functor~$\Hgr$ with field coefficients. If we fix the homology theory and the field of coefficients, and vary the homological dimension, then it turns out~\cite{Zigzagpersistenthomologyandreal} that the resulting collection of barcodes captures all the information that we are seeking in the present work.
There are four types of bars identified in~\cite{ExtendingPersistence}, each having a different geometric significance; this is explored in some detail by Bendich et al.~\cite{bendich}, as part of a broader program to understand homological stability of the fibers of an $\RR$-space. Two of of the four types can be matched to the standard persistence of $(X,f)$ and $(X,-f)$. The other two types provide new information.

The symmetry of this theory is, however, not at all obvious: there is no immediately manifest relationship between the extended persistence barcodes of $(X,f)$ and $(X,-f)$. The existence of such a symmetry was conjectured by Cohen-Steiner et al.~\cite{ExtendingPersistence} on the basis of results obtained for closed manifolds using duality theorems. The matter was resolved in~\cite{Zigzagpersistenthomologyandreal}, which establishes a precise symmetry between the two sets of barcodes, via calculations in zigzag persistent homology. The symmetry requires considering homology in more than one dimension at once, since the correspondence between the barcodes involves dimension shifts.

Finally, we note that it is relatively straightforward to use rectangle measures to generalize extended persistence to the continuous case; the procedure is outlined in~\cite{Thestructureandstability}. We will say more about extended persistence in Section~\ref{sec:extended-persistence}.
\end{example}

\begin{example}[interval persistent homology]
Dey and Wenger~\cite{StabilityCP} proposed a theory of `interval persistence'. They consider {interlevelsets} $\XX_a^b$,
%\[
%    \XX_a^b = f^{-1}[a,b].
%\]
seeking maximal intervals $[a, b]$ such that the sequence
\[
    \Hgr(\XX_a^a) \to \Hgr(\XX_a^{b-\eps}) \to \Hgr(\XX_a^b),
\]
supports a summand over the first two vector spaces, but not the third. In other words, they look for classes in the levelsets that vanish in interlevelsets. Although interval persistent homology still does not satisfy all our desired properties, it does suggest additional homological information that we want to recover from an $\RR$-space.
\end{example}

\begin{remark}
Building on this work, Burghelea, Dey and Haller have developed an analogous program to study the persistent homology of spaces fibered over the circle \cite{Burghelea_Dey_2013,Burghelea_Haller_2013}.
\end{remark}

Extended and interval persistence hint at what we mean by `all obvious homological information', and invite us to adopt a categorical perspective. Let $\Int$ denote the category of closed intervals $[a,b]$ in the real line; the morphisms are the inclusions $[a,b] \subseteq [c,d]$. Then an $\RR$-space $\XX = (X,f)$ can be thought of as a functor $\XX: \Int \to \Top$ that carries each interval $[a,b]$ to the corresponding interlevelset $\XX_a^b$; the morphism associated to an inclusion $[a,b] \subseteq [c,d]$ is the inclusion $\XX_a^b \subseteq \XX_c^d$. We are interested, then, in understanding the composite functors
\[
\begin{tikzpicture}[xscale=1.5,yscale=1]
\draw (0,0) node(00){$\Int$}  ;
\draw (1.5,0) node(10) {$\Top$} ;
\draw (3,0) node(02) {$\Vect$} ;
\draw[->] (00) to node[above]{$\XX$}  (10); 
\draw[->] (10) to node[above]{$\Hgr$}  (02);
\end{tikzpicture}
\]
where $\Hgr$ is a homology functor with coefficients in a field, and $\Vect$ is the category of vectors spaces over that field.

The following can be viewed as a preliminary attempt to understand this functor:

\begin{example}[levelset zigzag persistent homology]
In~\cite{Zigzagpersistenthomologyandreal}, Carlsson et al.\ proposed the following protocol for studying an $\RR$-space $\XX = (X,f)$. Suppose $\XX$ is Morse-like, with critical values $a_1 < a_2 < \dots < a_n$. Let $s_0 < s_1 < \dots < s_n$ be a collection of `intercritical values', interleaved between the critical values in the sense that $s_{i-1} < a_i < s_i$. Then the zigzag diagram of topological spaces (and inclusion maps)
\[
\begin{tikzpicture}[xscale=1.2,yscale=1.2]
\draw
(1,1) node(11) {$\XX_{s_0}^{s_1}$}
(3,1) node(31) {$\XX_{s_1}^{s_2}$}
(5,1) node(51) {$\cdots$}
(7,1) node(71) {$\XX_{s_{n-1}}^{s_n}$}
;
\draw
(0,0) node(00) {$\XX_{s_0}^{s_0}$}
(2,0) node(20) {$\XX_{s_1}^{s_1}$}
(4,0) node(40) {$\XX_{s_2}^{s_2}$}
(6,0) node(60) {$\cdots$}
(8,0) node(80) {$\XX_{s_n}^{s_n}$}
;
\draw[->] (00) -- (11); 
\draw[->] (20) -- (11);
\draw[->] (20) -- (31);
\draw[->] (40) -- (31);
\draw[->] (40) -- (51);
%\draw[->] (60) -- (51);  
\draw[->] (60) -- (71);
\draw[->] (80) -- (71);  
\end{tikzpicture}
\]
gives rise to a zigzag diagram of vector spaces (and linear maps)
\[
\begin{tikzpicture}[xscale=1.2,yscale=1.2]
\draw
(1,1) node(11) {$\Hgr(\XX_{s_0}^{s_1})$}
(3,1) node(31) {$\Hgr(\XX_{s_1}^{s_2})$}
(5,1) node(51) {$\cdots$}
(7,1) node(71) {$\Hgr(\XX_{s_{n-1}}^{s_n})$}
;
\draw
(0,0) node(00) {$\Hgr(\XX_{s_0}^{s_0})$}
(2,0) node(20) {$\Hgr(\XX_{s_1}^{s_1})$}
(4,0) node(40) {$\Hgr(\XX_{s_2}^{s_2})$}
(6,0) node(60) {$\cdots$}
(8,0) node(80) {$\Hgr(\XX_{s_n}^{s_n})$}
;
\draw[->] (00) -- (11); 
\draw[->] (20) -- (11);
\draw[->] (20) -- (31);
\draw[->] (40) -- (31);
\draw[->] (40) -- (51);
%\draw[->] (60) -- (51);  
\draw[->] (60) -- (71);
\draw[->] (80) -- (71);  
\end{tikzpicture}
\]
whose indecomposable summands are recorded as the levelset zigzag barcode of~$\XX$. There are four types of bars, according as the ends of the summand lie in the top row or the bottom row of the diagram. Each bar is then associated with an open, closed, or half-closed real interval with endpoints in the set of critical values; the Morse-like assumption ensures that the interval is precisely the interval of persistence of the corresponding homological feature. 

In~\cite{Zigzagpersistenthomologyandreal} it is shown that the levelset zigzag barcode carries exactly the same information as the extended persistence barcodes of $(X,f)$ and of $(X,-f)$, as well as another related object called the `up-down persistence' barcode. The advantage of levelset zigzag over the other, equivalent, theories is that it is manifestly symmetrical with respect to symmetries of the real line. Moreover, fiberwise homological features are expressed in the correct dimension in this theory; no dimension shifts take place.
\end{example}

The main weakness of levelset zigzag persistence is that it is stubbornly discrete, in the sense that it is a forbidding prospect to try to take a continuous limit of the zigzag diagrams used in the theory.
Parametrized homology is our response to this weakness.
We take advantage of the theory of rectangle measures from~\cite{Thestructureandstability} to define four continuous-parameter persistence diagrams, corresponding to the four types of bars in the levelset zigzag barcode. Each diagram represents a set of homological features and carries information about how they perish at both ends of the interval over which they are defined.
The diagrams are stable with respect to perturbation of the function~$f$.

One advantage of using rectangle measures is that the proofs, in a certain sense, become `bounded'. In the levelset zigzag framework, in order to prove anything, one has to consider zigzag diagrams of arbitrary length. In the parametrized homology framework,  result can be expressed as statements about rectangle measures, and can be proved using specific diagrams of a fixed size. The proofs are generally very straightforward, once the appropriate `diagram calculus' has been mastered.

%-----------------------------------------------------------------
\subsection*{Outline}
In Section~\ref{sec:tools}, we review the algebraic machinery needed to define parametrized homology: zigzag modules, quiver representation diagrams, rectangle measures.

Section~\ref{sec:PH} comprises the main body of this paper.
In Section~\ref{subsec:4-measures}, we provisionally define four rectangle measures that will eventually yield the four persistence diagrams of parametrized homology. A certain homological tautness property is required for these measures to be additive; this is treated in Sections \ref{sec:tautness} and~\ref{sec:additivity}. Section~\ref{sec:finiteness} identifies conditions under which the measures are finite. Under these favourable conditions, the construction of the four persistence diagrams in Section~\ref{paramhom} is immediate. In Section~\ref{subsec:LZZ}, we show that parametrized homology exactly emulates levelset zigzag persistence in the discrete Morse-like case. Section~\ref{sec:16} is devoted to geometric considerations. Each of the four diagrams may contain features supported over open, closed and half-open intervals. We illustrate the sixteen possible behaviours, and show that only four of them occur in the compact case. Finally, in Section~\ref{subsec:stability} we prove the stability theorem, and in Section~\ref{sec:extended-persistence} establish the relationship with continuous-parameter extended persistence.

A brief discussion of parametrized cohomology, in Section~\ref{sec:parcoho}, concludes the paper.
\section{Algebraic Tools}
\label{sec:tools}

In this section, we review the tools from \cite{ZigzagPersistence} and~\cite{Thestructureandstability} that we use to develop parametrized homology invariants. Throughout this paper, vector spaces are taken to be over an arbitrary field $\mathbf k$. In certain instances, the field is specified.

%-----------------------------------------------------------------
\subsection{Zigzag modules}
\label{subsec:zigzag}

A \emph{zigzag module} $\VV$ of length~$n$ (see \cite{ZigzagPersistence}) is a sequence of vector spaces and linear maps between them
\[
\begin{tikzpicture}[xscale=1.2,yscale=1.2]
\draw (1,1) node(11) {$V_1$} (1.97,1) node(21) {$V_2$} (2.94,1) node(31){$\ldots$} (4,1) node(41) {$V_n.$}(4.15,1);

\draw[<->] (11) to node[above]{$$}  (21); 
\draw[<->] (21)  to node[above]{$$}  (31); 
\draw[<->] (31)  to node[above]{$$}  (41);
\end{tikzpicture}
\]
Each
\!\!\!\!\!
\begin{tikzpicture}[xscale=0.8,yscale=1]
\draw (1,1) node(11){} (2,1) node(21){};
\draw[<->] (11) --  (21); 
\end{tikzpicture}
\!\!\!\!\!
represents either a forward map
\!\!\!\!\!
\begin{tikzpicture}[xscale=0.8,yscale=1]
\draw (1,1) node(11){} (2,1) node(21){};
\draw[->] (11) --  (21); 
\end{tikzpicture}
\!\!\!\!\!
or a backward map
\!\!\!\!\!
$\begin{tikzpicture}[xscale=0.8,yscale=1]
\draw (1,1) node(11){} (2,1) node(21){};
\draw[<-] (11) --  (21); 
\end{tikzpicture}$
\!\!\!\!\!.
The particular choice of directions for a given zigzag module is called its \emph{shape}. If every map is a forward map the zigzag module is called a \emph{persistence module}~\cite{ZC}.

The basic building blocks of zigzag modules are the interval modules. Fix a shape of length~$n$. The interval module $\II{[p,q]}$ of that shape is the zigzag module
\[
\begin{tikzpicture}[xscale=1.2,yscale=1.2]
\draw (1.05,1) node(11) {$I_1$} (1.97,1) node(21) {$I_2$} (2.93,1) node(31){$\ldots$} (3.95,1) node(41) {$I_n$}(4.15,1);

\draw[<->] (11) to node[above]{$$}  (21); 
\draw[<->] (21)  to node[above]{$$}  (31); 
\draw[<->] (31)  to node[above]{$$}  (41);
\end{tikzpicture}
\]
where $I_i =\mathbf k$ for $p \leq i \leq q$, and $I_i =0$ otherwise, and where every
{\raisebox{-0.28em}{\begin{tikzpicture}[xscale=1,yscale=1]
\draw (1,0) node(11){$\mathbf k$} (2,0) node(21){$\mathbf k$};
\draw[->] (11) --(21); 
\end{tikzpicture}}}
or
{\raisebox{-0.28em}{\begin{tikzpicture}[xscale=1,yscale=1]
\draw (1,0) node(11){$\mathbf k$} (2,0) node(21){$\mathbf k$};
\draw[<-] (11) --(21); 
\end{tikzpicture}}}
is the identity map.

\begin{example}
Let $\VV_{\{1, 2, 3\}} = {\raisebox{-0.5em}{\begin{tikzpicture}[xscale=1.15,yscale=1.15]
\draw (1,1) node(11) {$V_1$} (2,1) node(21) {$V_2$} (3,1) node(31){$V_3$};

\draw[->] (11) to node[above]{$$}  (21); 
\draw[->] (21)  to node[above]{$$}  (31); 
\end{tikzpicture}}}$.  The six interval modules over $\VV$ may be represented pictorially as follows:
%
%\[
%\begin{array}{cc}
%\begin{array}{lcl}
%\II[1,1] &=&  \onof{1}  \\
%\II[2,2] &=&   \onof{2} \\
%\II[3,3] &=&  \onof{3} \\
%\end{array}
%&
%\begin{array}{lcl}
%\II[1,2] &=&   \onof{4} \\
%\II[2,3] &=&   \onof{5} \\
%\II[1,3] &=&   \onof{6} \\
%\end{array}\\
%\end{array}
%\]
%
\begin{alignat*}{3}
\II[1,3] &= \onof{6}
\qquad
&
\II[2,3] &= \onof{5}
\qquad
&
\II[3,3] &= \onof{3}
\\
\II[1,2] &= \onof{4}
\qquad
&
\II[2,2] &= \onof{2}
\\
\II[1,1] &= \onof{1}
\end{alignat*}
Each dark green node represents a copy of the field~$\kk$ and each light pink node represents a copy of the zero vector space. Identity maps are represented by thickened green lines.
\end{example}
 
A theorem of Gabriel~\cite{quiver} implies that any finite-dimensional zigzag module can be decomposed as a direct sum of interval modules. The extension to infinite-dimensional zigzag modules follows from a theorem of Auslander~\cite{Auslander}. The list of summands that appear in the decomposition is an isomorphism invariant of~$\VV$ by the Krull--Schmidt--Azumaya theorem~\cite{Azumaya}. We call this isomorphism invariant the \emph{zigzag persistence} of~$\VV$.

\begin{example}\label{zigps}
Consider a zigzag diagram $\XX$ of topological spaces and continuous maps between them:
\[
\begin{tikzpicture}[xscale=1.25,yscale=1.2]
\draw (1,1) node(11) {$X_1$} (2,1) node(21) {$X_2$} (3,1) node(31){$\ldots$} (4,1) node(41) {$X_n$};

\draw[<->] (11) to node[above]{$$}  (21); 
\draw[<->] (21)  to node[above]{$$}  (31); 
\draw[<->] (31)  to node[above]{$$}  (41);
\end{tikzpicture}
\]
We get a zigzag module $\Hgr\XX$ by applying a homology functor $\Hgr = \Hgr_j(-; \kk)$ to this diagram.
Decomposing the diagram, we can write
\[
\begin{tikzpicture}[xscale=1.63,yscale=1.2]
\draw
% (-0.25, 1) node(01){$\Hgr_j(\mathcal{X})$ \;:}
(0.75,1) node(11) {$\Hgr_j(X_1)$} (2,1) node(21) {$\Hgr_j(X_2)$} (3,1) node(31){$\ldots$} (4,1) node(41) {$\Hgr_j(X_n)$} (4.85,1) node(51){$\cong$} (6,1) node(61){$\bigoplus_{i\in I}\II[p_i, q_i].$};
\draw[<->] (11) to node[above]{$$}  (21); 
\draw[<->] (21)  to node[above]{$$}  (31); 
\draw[<->] (31)  to node[above]{$$}  (41);
\end{tikzpicture}
\]
The zigzag persistent homology of $\XX$ (for the functor~$\Hgr$) is then the multiset of intervals $[p_i,q_i]$ in the interval decomposition.
\end{example}

\begin{definition}
The \emph{multiplicity} of an interval $[p,q]$ in a zigzag module~$\VV$ is the number of copies of $\II[p,q]$ that occur in the interval decomposition of $\VV$. This number is written
\[
\langle [p,q] \mid \VV \rangle
\]
and takes values in the set $\{ 0, 1, 2, \ldots, \infty \}$. (For our purposes we do not need to distinguish different infinite cardinals.)
Finally, the \emph{persistence diagram} of~$\VV$ is the multiset
\[
\Dgm(\VV)
\quad \text{in} \quad
\{ (p,q) \mid 1 \leq p \leq q \leq n \}
\]
defined by the multiplicity function $(p,q) \mapsto \langle [p,q] \mid \VV \rangle$.
\end{definition}

We will often use pictorial notation for these multiplicities. For example, given a persistence module
$
\VV = 
 {\raisebox{-0.5em}{\begin{tikzpicture}[xscale=1.2,yscale=1.2]
\draw (1,1) node(11) {$V_1$} (2,1) node(21) {$V_2$} (3,1) node(31){$V_3$};
\draw[->] (11) to node[above]{$$}  (21); 
\draw[->] (21)  to node[above]{$$}  (31); 
\end{tikzpicture}}}
$
we may write
\[
\langle [2,3] \mid \VV \rangle
\quad \text{or} \quad
\langle \onof{5} \mid \VV \rangle
\quad \text{or simply} \quad
\langle \onof{5} \rangle
\]
for the multiplicity of $\II[2,3]$ in~$\VV$.

%-----------------------------------------------------------------
\subsection{Two calculation principles}
\label{subsec:2calc}

There are two methods from~\cite{ZigzagPersistence} that we repeatedly use to calculate multiplicities: the Restriction Principle and the Diamond Principle.

\begin{theorem}[Restriction Principle]\label{RP}
Let $\VV$ be a zigzag module with two consecutive maps in the same direction
\[
\begin{tikzpicture}[xscale=1.2,yscale=1]
\draw (0.85,1) node(11) {$V_1$} (1.9,1) node(21) {$V_2$} (2.95,1) node(31){$\ldots$} (4.15,1) node(41) {$V_{k-1}$} (5.30,1) node(51){$V_k$}  (6.30,1) node(61){$V_{k+1}$} (7.35,1) node(71){$\ldots$} (8.45,1) node(81){$V_n$}  ;
\draw[<->] (11) to node[above]{$$}  (21); 
\draw[<->] (21)  to node[above]{$$}  (31); 
\draw[<->] (31)  to node[above]{$$}  (41);
\draw[->] (41) to node[above]{$\scriptstyle g$}  (51); 
\draw[->] (51)  to node[above]{$\scriptstyle h$}  (61); 
\draw[<->] (61)  to node[above]{$$}  (71);
\draw[<->] (71)  to node[above]{$$}  (81);
\end{tikzpicture}
\]
and let $\WW$ be the zigzag module
\[
\begin{tikzpicture}[xscale=1.2,yscale=1]
\draw (0.85,1) node(11) {$V_1$} (1.9,1) node(21) {$V_2$} (2.95,1) node(31){$\ldots$} (4.15,1) node(41) {$V_{k-1}$} (6.30,1) node(61){$V_{k+1}$} (7.35,1) node(71){$\ldots$} (8.45,1) node(81){$V_n$}  ;
\draw[<->] (11) to node[above]{$$}  (21); 
\draw[<->] (21)  to node[above]{$$}  (31); 
\draw[<->] (31)  to node[above]{$$}  (41);
\draw[->] (41) to node[above]{$\scriptstyle hg$}  (61); 
\draw[<->] (61)  to node[above]{$$}  (71);
\draw[<->] (71)  to node[above]{$$}  (81);
\end{tikzpicture}
\]
obtained by combining those maps into a single composite map and deleting the intermediate vector space~$V_k$. Let $[p,q]$ be an interval over the index set for~$\WW$ (so $p,q \ne k$).
Then
\[
\langle [p, q]\,|\, \WW \rangle = \sum_{[\hat{p}, \hat{q}]} \langle [\hat{p}, \hat{q}]\,|\, \VV \rangle
\]
where the sum is over those intervals $[\hat{p},\hat{q}]$ over the index set for~$\VV$ that restrict to $[p,q]$ over the index set of $\WW$.
\end{theorem}

\begin{proof}
Take an arbitrary interval decomposition of $\VV$. This induces an interval decomposition of $\WW$. Summands of $\WW$ of type $[p,q]$ arise precisely from summands of $\VV$ of types $[\hat{p},\hat{q}]$ that restrict to $[p, q]$ over the index set of $\WW$.
\end{proof}

\begin{example}
Consider a zigzag module
\[
\VV = 
{\raisebox{-0.5em}{\begin{tikzpicture}[xscale=1.1,yscale=1]
\draw (1,1) node(11) {$V_1$} (2,1) node(21) {$V_2$} (3,1) node(31){$V_3$}(4, 1) node(41){$V_4$} (5, 1)node(51){$V_5$};
\draw[->] (11) to node[above]{$$}  (21); 
\draw[->] (21)  to node[above]{$$}  (31); 
\draw[<-] (31)  to node[above]{$$}  (41); 
\draw[<-] (41)  to node[above]{$$}  (51);
\end{tikzpicture}}}
\]
and its restrictions
\begin{align*}
\VV_{1,2,3,5}
&=
{\raisebox{-0.5em}{\begin{tikzpicture}[xscale=1.1,yscale=1]
\draw
(1,1) node(11) {$V_1$}
(2,1) node(21) {$V_2$}
(3,1) node(31){$V_3$}
(5,1) node(51){$V_5$};
\draw[->] (11) to node[above]{$$}  (21); 
\draw[->] (21)  to node[above]{$$}  (31); 
\draw[<-] (31)  to node[above]{$$}  (51); 
\end{tikzpicture}}}
\\
\VV_{1,3,4,5}
&=
{\raisebox{-0.5em}{\begin{tikzpicture}[xscale=1.1,yscale=1]
\draw
(1,1) node(11) {$V_1$}
(3,1) node(31){$V_3$}
(4,1) node(41){$V_4$}
(5,1) node(51){$V_5$};
\draw[->] (11) to node[above]{$$}  (31); 
\draw[<-] (31)  to node[above]{$$}  (41); 
\draw[<-] (41)  to node[above]{$$}  (51); 
\end{tikzpicture}}}
\end{align*}
obtained in the manner described above.
Then
\begin{align*}
\langle \eksample{2} \mid \VV_{1,2,3,5} \, \rangle
&=
\langle \eksample{1} \mid \VV \, \rangle,
\\
\langle \eksample{4} \mid \VV_{1,3,4,5} \, \rangle
&=
\langle \eksample{1} \mid \VV \, \rangle
+
\langle \eksample{6} \mid \VV \, \rangle.
\end{align*}
The extra term occurs when the interval for the restricted module abuts the long edge on either side (so there is both a clear node and a filled node at that edge). There are then two possible intervals which restrict to it.
\end{example}

The Diamond Principle relates the interval multiplicities of zigzag modules that are related by a different kind of local change.
The principle is most sharply expressed in terms of the reflection functors of Bernstein, Gelfand and Ponomarev~\cite{reflection}. We make do with a simpler non-functorial statement. We say that a diamond-shaped commuting diagram of vector spaces
\[
\begin{tikzpicture}[xscale=1,yscale=1]
\draw (1.5,1.5) node(11) {$C$} (-1.5,1.5) node(011) {$B$} (0,3) node(02) {$D$} ;
\draw (0,0) node(00){$A$} ;
\draw[->] (00) to node[below]{$~i_2$}  (11); 
\draw[->] (00) to node[below]{$i_1$}  (011);
\draw[->] (011) to node[left]{$j_1~$}  (02);
\draw[->] (11) to node[right]{$~j_2$}  (02);
\end{tikzpicture}
\]
is \emph{exact} if the sequence
\[
\begin{tikzpicture}[xscale=1.8,yscale=1]
\draw (1.5,0) node(10) {$B\oplus C$} (3,0) node(02) {$D$} ;
\draw (0,0) node(00){$A$}  ;
\draw[->] (00) to node[above]{{\small$i_1\oplus i_2$}}  (10); 
\draw[->] (10) to node[above]{$j_1-j_2$}  (02);
\end{tikzpicture}
\]
is exact at $B \oplus C$. This means that a pair of vectors $\beta \in B$, $\gamma \in C$ satisfies $j_1(\beta) = j_2(\gamma)$ if and only if there exists $\alpha \in A$ such that $\beta = i_1(\alpha)$ and $\gamma  = i_2(\alpha)$.

\begin{theorem}[Diamond Principle~\cite{ZigzagPersistence}]
\label{DP}
Consider a diagram of vector spaces
\begin{center}
\begin{tikzpicture}[xscale=1,yscale=1]

\draw (-2.3,0) node(00){$V_1$} ;
\draw (-1,0) node(10){$\ldots$} ;
\draw (0.5,0) node(110){$V_{k-2}$} ;
\draw (2,0) node(20){$V_{k-1}$} ;
\draw (3,1) node(31){$V_k^+$} ;
\draw (3,-1) node(301){$V_k^-$};
\draw (4,0) node(40){$V_{k+1}$};
\draw (5.5,0) node(550){$V_{k+2}$};
\draw (7,0) node(50){$\ldots$};
\draw (8.3,0) node(60){$V_{n}$};

\draw[<->] (00) to node[above]{$$} (10); 
\draw[<->] (10) to node[above]{$$}(110);
\draw[<->] (110) to node[above]{$$} (20);
\draw[->] (20) to node[left]{} (31);
\draw[->] (40) to node[right]{} (31); 
\draw[->] (301) to node[left]{} (20);
\draw[->] (301) to node[right]{} (40); 
\draw[<->] (40) to node[above]{$$} (550); 
\draw[<->] (550) to node[above]{$$} (50); 
\draw[<->] (50) to node[above]{$$} (60); 

\end{tikzpicture}
\end{center}
where the middle diamond is exact.
Let $\VV^+, \VV^-$ respectively denote the upper zigzag module (containing $V_k^+$) and the lower zigzag module (containing $V_k^-$) in this diagram. Then the following multiplicities are equal.

(i) If the interval $[p,q]$ does not meet $\{k-1, k, k+1\}$ then
\[
\langle [p,q] \mid \VV^+ \rangle
=
\langle [p,q] \mid \VV^- \rangle.
\]

(ii) If the interval $[p,q]$ completely contains $\{k-1, k, k+1\}$ then
\[
\langle [p,q] \mid \VV^+ \rangle
=
\langle [p,q] \mid \VV^- \rangle.
\]

(iii) For $p \leq k-1$ we have
\begin{alignat*}{4}
&\langle [p,k] &&\mid \VV^+ \rangle
&&= \langle [p,k-1] &&\mid \VV^- \rangle,
\\
&\langle [p,k-1] &&\mid \VV^+ \rangle
&&= \langle [p,k] &&\mid \VV^- \rangle.
\end{alignat*}

(iv) For $q \geq k+1$ we have
\begin{alignat*}{4}
&\langle [k,q] &&\mid \VV^+ \rangle
&&= \langle [k+1,q] &&\mid \VV^- \rangle,
\\
&\langle [k+1,q] &&\mid \VV^+ \rangle
&&= \langle [k,q] &&\mid \VV^- \rangle.
\end{alignat*}

\noindent
The diagrams
\begin{center}
%\scalebox{0.1}{\input{diamondmv.tex}}
\raisebox{2pc}{(ii)}
\includegraphics[scale=2.5,page=5]{diamond-principle}
\qquad
\raisebox{2pc}{(iii)}
\includegraphics[scale=2.5,page=2]{diamond-principle}
\includegraphics[scale=2.5,page=1]{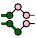}
\qquad
\raisebox{2pc}{(iv)}
\includegraphics[scale=2.5,page=4]{diamond-principle}
\includegraphics[scale=2.5,page=3]{diamond-principle}
\end{center}
express the last three of these rules pictorially.
\qed
\end{theorem}

\begin{remark}
The theorem gives no information about $\langle [k,k] \mid \VV^+ \rangle$ or $\langle [k,k] \mid \VV^- \rangle$. These quantities are independent of each other and of all other multiplicities.
\end{remark}

We use the Diamond Principle frequently in the following situation. Consider a diagram of topological spaces of the following form:
\begin{center}
\begin{tikzpicture}[xscale=1,yscale=1]

\draw (-2.3,0) node(00){$X_1$} ;
\draw (-1,0) node(10){$\ldots$} ;
\draw (0.5,0) node(110){$X_{k-2}$} ;
\draw (2,0) node(20){$A$} ;
\draw (3,1) node(31){$A \cup B$} ;
\draw (3,-1) node(301){$A\cap B$};
\draw (4,0) node(40){$B$};
\draw (5.5,0) node(550){$X_{k+2}$};
\draw (7,0) node(50){$\ldots$};
\draw (8.3,0) node(60){$X_{n}$};

\draw[<->] (00) to node[above]{$$} (10); 
\draw[<->] (10) to node[above]{$$}(110);
\draw[<->] (110) to node[above]{$$} (20);
\draw[->] (20) to node[left]{} (31);
\draw[->] (40) to node[right]{} (31); 
\draw[->] (301) to node[left]{} (20);
\draw[->] (301) to node[right]{} (40); 
\draw[<->] (40) to node[above]{$$} (550); 
\draw[<->] (550) to node[above]{$$} (50); 
\draw[<->] (50) to node[above]{$$} (60); 

\end{tikzpicture}
\end{center}
Here $A,B$ are subspaces of some common ambient space. Applying a homology functor $\Hgr$, we obtain an upper zigzag diagram~$\VV^\cup$ and a lower zigzag diagram~$\VV^\cap$. The exactness of the diamond is precisely the exactness of the central term in the following excerpt from the Mayer--Vietoris sequence:
\[
\begin{tikzpicture}[xscale=3,yscale=1]
\draw
(0.3,1) node(01) {$\ldots$}
(1,1) node(11) {$ \Hgr(A \cap B)$}
(2,1) node(21) {$ \Hgr(A)\oplus  \Hgr(B)$}
(3,1) node(31){$ \Hgr(A \cup B)$}
(3.7,1) node(41) {$\ldots$};

\draw[->] (01) to node[above]{$$}  (11); 
\draw[->] (11) to node[above]{$$}  (21); 
\draw[->] (21)  to node[above]{$$}  (31); 
\draw[->] (31)  to node[above]{$$}  (41);
%\draw[->] (41)  to node[above]{$$}  (51);
\end{tikzpicture}
\]
In situations where the Mayer--Vietoris theorem holds, we can use the Diamond Principle to compare the interval summands of $\VV^\cup$ and~$\VV^\cap$. The reader is reminded that the Mayer--Vietoris theorem is not always applicable. We treat this matter carefully in Section~\ref{sec:tautness}.

%-----------------------------------------------------------------
\subsection{Persistence diagrams and measures}
\label{subsec:measures}

As we discussed in Section 2.1, a zigzag module with a finite index set decomposes into interval modules, the list of summands being unique up to reordering. There are finitely many interval module types, so the structure of the zigzag module is determined by a finite list of multiplicities.

On the other hand, the objects we are studying are spaces parametrized over the real line; and so we will want to define continuous-parameter persistence diagrams. The motivating heuristic is that each topological feature will be supported over some interval of~$\RR$. These intervals may be open, closed or half-open, so we follow Chazal et al.~\cite{Thestructureandstability} in describing their endpoints as real numbers \emph{decorated} with a $^+$~or~$^-$ superscript. The superscript~$^*$ may be used for an unspecified decoration. Here are the four options:

\medskip
\begin{center}
\begin{tabular}{|c|c|c|}
\hline
\footnotesize{interval}
&
\footnotesize{decorated pair}
&
\footnotesize{point with tick}
\\
\hline
\hline
$(p,q)$ & $(p^+,q^-)$
& \raisebox{-0.6em}{\includegraphics[scale=1.25]{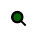}}
\\
\hline
$(p,q]$ & $(p^+,q^+)$
& \raisebox{-0.6em}{\includegraphics[scale=1.25]{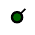}}
\\
\hline
$[p,q)$ & $(p^-,q^-)$
& \raisebox{-0.6em}{\includegraphics[scale=1.25]{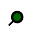}}
\\
\hline
$[p,q]$ & $(p^-,q^+)$
& \raisebox{-0.6em}{\includegraphics[scale=1.25]{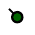}}
\\
\hline
\end{tabular}
\end{center}
\medskip

\noindent
Except for the degenerate interval $[p, p] = (p^-, p^+)$, we require $p < q$. For infinite intervals, we allow $p = -\infty$ and $q = +\infty$ and their decorated forms $p^* = -\infty^+$ and $q^* = +\infty^-$.

Given a collection (i.e.\ multiset) of such intervals, we can form a persistence diagram by drawing each $(p^*,q^*)$ as a point in the plane with a tick to indicate the decorations. The tick convention is self-explanatory. The diagram resides in the extended half-plane
\[
\Hp =\{ (p,q)  \mid -\infty \leq p < q \leq \infty \}
\]
which we can draw schematically as a triangle.
If we omit the ticks (i.e.\ forget the decorations), what remains is an \emph{undecorated} persistence diagram.

Our main mechanism for defining and studying continuous-parameter persistence modules is taken from~\cite{Thestructureandstability}: a finite measure theory designed for this task.
Define
\[
\Rect(\Hp) = \{[a, b]\times [c,d]\subset \Hp \, |\, -\infty \leq a < b < c < d \leq +\infty\}.
\]
This consists of finite rectangles, horizontal semi-infinite strips, vertical semi-infinite strips and infinite quadrants in~$\Hp$.
A rectangle measure or \emph{r-measure} on $\Hp$ is a function 
\[
\begin{tikzpicture}[xscale=1.3,yscale=1]
\draw (0.1,0)node(11) {$\mu \colon \Rect(\Hp)$} (3,0) node(21){$ \{0, 1, 2, 3, \ldots \}\cup \{\infty \}$};
\draw[->] (11) -- (21); 
\end{tikzpicture}
\]
that is additive with respect to splitting a rectangle horizontally or vertically into two rectangles. Explicitly, we require
\begin{alignat*}{2}
\mu([a,b] \times [c,d])
&=
\mu([a,p] \times [c,d])
+
\mu([p,b] \times [c,d])
&&
\qquad \text{\footnotesize (horizontal split)}
\\
\mu([a,b] \times [c,d])
&=
\mu([a,b] \times [c,q])
+
\mu([a,b] \times [q,d])
&&
\qquad \text{\footnotesize (vertical split)}
\end{alignat*}
whenever $a < p < b < c < q < d$ (see Figure~\ref{fig1}). By iterating these formulas, it follows that $\mu$ must be additive with respect to arbitrary tilings of a rectangle by other rectangles. This implies, in particular, that $\mu$ is monotone with respect to inclusion of rectangles.

\begin{figure}[h!]
\includegraphics[scale=1.2]{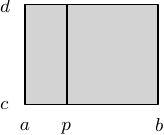}
\qquad\qquad
\includegraphics[scale=1.2]{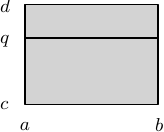}
\label{fig1}\caption{Rectangles split horizontally and vertically.}
\end{figure}

The `atoms' for this measure theory are decorated points rather than points; when a rectangle is split in two, points along the split line have to be assigned to one side or the other and this is done using the tick. We write $(p^*,q^*) \in R$ to mean that $(p,q)$ lies in~$R$ with the tick pointing into the interior of~$R$ (this is automatic for interior points).

\medskip
\begin{figure}[h!]
\includegraphics[scale=1.2]{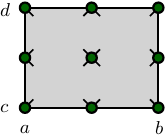} 
\caption{A decorated point $(p^*,q^*)$ is contained in $R$ if and only if $(p, q)$ is contained in $R$ and the tick points into the interior.}
\end{figure}

\begin{theorem}[{\cite[Theorem 3.12]{Thestructureandstability}}]
\label{equivextended}
There is a bijective correspondence between
\begin{itemize}
\item
Finite r-measures $\mu$ on $\Hp$; and
\smallskip
\item
Locally finite multisets $A$ of decorated points in $\Hp$.
\end{itemize}
Here `finite' means that ${\mu(R)<\infty}$ for all~$R$, and `locally finite' means that ${\card(A|_{R})<\infty}$ for all~$R$.
Explicitly, a multiset~$A$ corresponds to the measure~$\mu$ defined by the formula
\[
\mu(R) = \card(A|_{R}),
\]
(the cardinality of the multiset of decorated points of~$A$ that belong to~$R$);
and, conversely, a measure~$\mu$ corresponds to the multiset~$A$ with multiplicity function
\[
\mathrm{m}_A(p^*, q^*)
=
\min \{\mu(R)\, \mid\,
R \in \Rect(\Hp) \;\text{\rm such that}\; (p^*, q^*) \in R \}.
\]
In other words, finite r-measures correspond exactly to decorated persistence diagrams.
\qed
\end{theorem}

\begin{remark}
Since r-measures are monotone, the `$\min$' in the formula for $\mathrm{m}_A$ can be calculated as a limit. For example
\begin{align*}
\mathrm{m}_A(p^+,q^-)
&=
\lim_{\epsilon \to 0} \mu([p, p+\epsilon] \times [q-\epsilon, q]),
\end{align*}
with similar formulas for the other choices of decoration for $(p^*,q^*)$ and for points at infinity. Since the expression inside the `lim' takes values in the natural numbers and decreases as $\epsilon$ decreases, it necessarily stabilizes for sufficiently small~$\epsilon$.
\end{remark}

The multiset~$A$ corresponding to a finite r-measure~$\mu$ is its \emph{decorated diagram}, written $\Dgm(\mu)$. We obtain the \emph{undecorated diagram} $\dgm(\mu)$ by forgetting the decorations. This is a multiset in~$\Hp$.

When the r-measure is not finite, the \emph{finite support} is defined in~\cite{Thestructureandstability} to be the set of decorated points in~$\Hp$ that are contained in some rectangle of finite measure. Within the finite support there is a well-defined decorated persistence diagram which characterizes the r-measure as above, with the proviso that rectangles which extend beyond the finite support have infinite measure. In particular, the undecorated diagram can be thought of as a locally finite multiset defined in some open set $\mathcal{F} \subseteq \Hp$ and deemed to have infinite multiplicity everywhere else in the extended plane.

%-----------------------------------------------------------------
%-----------------------------------------------------------------
\section{Parametrized Homology}
\label{sec:PH}

In this section we define `parametrized homology' invariants for $\RR$-spaces. Given an $\RR$-space $\XX = (\XX,f)$ and a homology functor $\Hgr$ with field coefficients, we define four persistence diagrams
\[
\Dgm\!\!{}^\uds (\Hgr \XX),\;
\Dgm\!\!{}^\dds (\Hgr \XX),\;
\Dgm\!\!{}^\uus (\Hgr \XX),\;
\Dgm\!\!{}^\dus (\Hgr \XX)
\]
that detect topological features exhibiting four different behaviors. We will need to impose conditions on $\Hgr$ and $\XX$ to guarantee that the r-measures used to define these diagrams are additive and finite.

%-----------------------------------------------------------------
\subsection{Four measures}
\label{subsec:4-measures}

Let $\XX = (X, f)$ be a $\RR$-space and let $\Hgr$ be a homology functor with field coefficients. Given a rectangle
\[
R = [a,b]\times [c,d],
\qquad
-\infty \leq a < b < c < d \leq +\infty,
\]
we wish to count the homological features of $\XX$ that are supported over the closed interval $[b, c]$ but do not reach either end of the open interval $(a, d)$.
Accordingly,  consider the diagram
\[
\XX_{\{a, b, c, d\}}:  \raisebox{-4.25ex}{
\begin{tikzpicture}[xscale=1.05,yscale=1.05]
\draw (1,1) node(11) {$\XX_a^b$} (3,1) node(31) {$\XX_b^c$} (5,1) node(51) {$\XX_c^d$} ;
\draw (0,0) node(00){$\XX_a^a$} (2,0) node(20){$\XX_b^b$} (4,0) node(40){$\XX_c^c$} (6,0) node(60){$\XX_d^d$}  ;
\draw[->] (00) -- (11); 
\draw[->] (20) -- (11);
\draw[->] (20) -- (31);
\draw[->] (40) -- (31);
\draw[->] (40) -- (51);
\draw[->] (60) -- (51);  
\end{tikzpicture}}
\]
of spaces and inclusion maps, where $\XX_a^b = f^{-1}[a, b]$. We assume $\XX_{-\infty}^{-\infty}$ and $\XX_{+\infty}^{+\infty}$ to be empty if they occur.
Apply~$\Hgr$ to obtain a diagram
\[
\Hgr\XX_{\{a, b, c, d\}}: 
\raisebox{-4.25ex}{
\begin{tikzpicture}[xscale=1.1,yscale=1.1]
\draw
(1,1) node(11) {$\Hgr(\XX_a^b)$}
(3,1) node(31) {$\Hgr(\XX_b^c)$}
(5,1) node(51) {$\Hgr(\XX_c^d)$} ;
\draw
(0,0) node(00){$\Hgr(\XX_a^a)$}
(2,0) node(20){$\Hgr(\XX_b^b)$}
(4,0) node(40){$\Hgr(\XX_c^c)$}
(6,0) node(60){$\Hgr(\XX_d^d)$}  ;
\draw[->] (00) -- (11); 
\draw[->] (20) -- (11);
\draw[->] (20) -- (31);
\draw[->] (40) -- (31);
\draw[->] (40) -- (51);
\draw[->] (60) -- (51);  
\end{tikzpicture}}
\]
of vector spaces and linear maps. Decomposing this zigzag module into interval modules, four of the multiplicities are of interest to us. Define four quantities as follows:
\begin{align*}
\muud_{\Hgr\XX}(R)
&= \langle\zzud \, |\, \Hgr\XX_{\{a, b, c, d\}} \rangle
\\
\mudd_{\Hgr\XX}(R)
&= \langle \zzdd \, |\, \Hgr\XX_{\{a, b, c, d\}} \rangle
\\
\muuu_{\Hgr\XX}(R)
&= \langle \zzuu \, |\, \Hgr\XX_{\{a, b, c, d\}} \rangle
\\
\mudu_{\Hgr\XX}(R)
&= \langle \zzdu \, |\, \Hgr\XX_{\{a, b, c, d\}} \rangle.
\end{align*}
Each of these counts topological features of a certain type, supported over $[b,c]$ but not outside $(a,d)$.
Under favorable circumstances, these four functions of~$R$ turn out to be finite r-measures and therefore their behavior can be completely described by a decorated persistence diagram in the extended half-space. We will identify such circumstances in later parts of this chapter.
 
The distinction between the four behaviors is seen in Figure~\ref{fig:four}. Consider 0-dimensional singular homology $\Hgr = \Hgr_0(-;\kk)$. In each example $\Hgr \XX_b^b \cong \Hgr\XX_b^c \cong \Hgr\XX_c^c$ have rank two whereas $\Hgr\XX_a^a$, $\Hgr\XX_d^d$ each have rank one. The way in which the second feature (i.e.\ the second connected component) perishes at each end is determined by the ranks of the maps
\[
\raisebox{-1.4ex}{\begin{tikzpicture}[xscale=1.6,yscale=1]
\draw (0,0) node(00){$\Hgr\XX_a^b$} (1,0) node(10){$\Hgr\XX_b^b$};
\draw[->] (10) -- (00); 
\end{tikzpicture}}
\quad\text{and}\quad
\raisebox{-1.4ex}{\begin{tikzpicture}[xscale=1.6,yscale=1]
\draw (0,0) node(00){$\Hgr\XX_c^c$} (1,0) node(10){$\Hgr\XX_c^d$};
\draw[->] (00) -- (10); 
\end{tikzpicture}}.
\]
If the rank is two, then the feature has simply \emph{expired} at that end: it is no longer there at $\XX_a^a$ or~$\XX_d^d$. If the rank is one, that means the feature has been \emph{killed} by some 1-cell that has appeared in $\XX_a^b$ or $\XX_c^d$.
In terms of zigzag summands, the situation looks like this:
\begin{alignat*}{2}
\raisebox{0.5ex}{\text{\scriptsize{is killed}}}
&\; \zzud \;
\raisebox{0.5ex}{\text{\scriptsize{is killed}}}
\qquad
&
\raisebox{0.5ex}{\text{\scriptsize{is killed}}}
&\; \zzuu \;
\raisebox{0.5ex}{\text{\scriptsize{expires}}}
\\
\raisebox{0.5ex}{\text{\scriptsize{expires}}}
&\; \zzdd \;
\raisebox{0.5ex}{\text{\scriptsize{is killed}}}
\qquad
&
\raisebox{0.5ex}{\text{\scriptsize{expires}}}
&\; \zzdu \;
\raisebox{0.5ex}{\text{\scriptsize{expires}}}
\end{alignat*}
Our definitions associate the four symbols $\udt, \ddt, \uut, \dut$ with these four behaviors. An unspecified behavior may be indicated by the symbol $\xxt$.

\begin{figure}
\includegraphics[scale = 1.1]{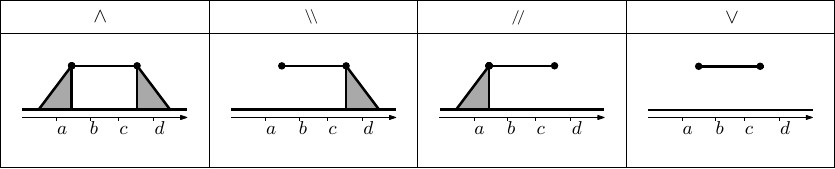}
\caption{
Two components (over $[b,c]$) become one (over~$a$ and~$d$). The four ways this can happen are detected by $\muud, \mudd, \muuu, \mudu$ respectively.}
\label{fig:four}
\end{figure}

\begin{proposition}
The four behaviors have `coordinate-reversal' symmetry. Specifically, suppose $\XX = (X,f)$ and $R = [a,b] \times [c,d]$. If we define the coordinate reversals $\overline\XX = (X,-f)$ and $\overline{R} = [-d,-c] \times [-b,-a]$ then the relations
\begin{alignat*}{2}
\muud_{\Hgr \overline{\XX}} (\overline{R}) &= \muud_{\Hgr\XX} (R)
\qquad&
\muuu_{\Hgr \overline{\XX}} (\overline{R}) &= \mudd_{\Hgr\XX} (R)
\\
\mudd_{\Hgr \overline{\XX}} (\overline{R}) &= \muuu_{\Hgr\XX} (R)
\qquad&
\mudu_{\Hgr \overline{\XX}} (\overline{R}) &= \mudu_{\Hgr\XX} (R)
\end{alignat*}
follow immediately.
\qed
\end{proposition}

Our next step is to identify when the four functions $\mu^\xxs_{\Hgr\XX}(R)$ are finite r-measures. We consider additivity first (Sections \ref{sec:tautness} and~\ref{sec:additivity}), then finiteness (Section~\ref{sec:finiteness}).

%-----------------------------------------------------------------
\subsection{Tautness}
\label{sec:tautness}

In proving additivity and other identities, we will make much use of the Diamond Principle. For $p < q < r < s$, consider the following diamonds:
\[
\raisebox{-8ex}{
\begin{tikzpicture}[xscale=1.2,yscale=1.2]
\draw
(1,1) node(11) {$\Hgr_k(\XX_p^s)$}
(0,0) node(00){$\Hgr_k(\XX_p^r)$}
(2,0) node(20){$\Hgr_k(\XX_q^s)$}
(1,-1) node(101){$\Hgr_k(\XX_q^r)$} ;
%----
\draw[->] (00) -- (11); 
\draw[->] (20) -- (11);
\draw[<-] (20) -- (101);
\draw[<-] (00) -- (101);
\end{tikzpicture}
}
\qquad \text{and} \qquad
\raisebox{-8ex}{
\begin{tikzpicture}[xscale=1.2,yscale=1.2]
\draw
(1,1) node(11) {$\Hgr_k(\XX_p^r)$}
(0,0) node(00){$\Hgr_k(\XX_p^q)$}
(2,0) node(20){$\Hgr_k(\XX_q^r)$}
(1,-1) node(101){$\Hgr_k(\XX_q^q)$} ;
%----
\draw[->] (00) -- (11); 
\draw[->] (20) -- (11);
\draw[<-] (20) -- (101);
\draw[<-] (00) -- (101);
\end{tikzpicture}
}
\]
The exactness of the left diamond is guaranteed by the Mayer--Vietoris theorem, which applies because the relative interiors of $\XX_p^r, \XX_q^s$ contain the sets $f^{-1}[p,r), f^{-1}(q,s]$ which cover $\XX_p^s$.
In contrast, there is no such guarantee for the right diamond: the relative interiors of $\XX_p^q, \XX_q^r$ do not cover $\XX_p^r$.

We identify a local condition on the embedding of~$\XX_q^q$ in~$\XX$, in terms of the homology theory~$\Hgr$, which gives us exactness of all such diamonds.
Let $U$ be any neighborhood of $\XX_q^q$ (such as $\XX_p^r$). It splits into two parts: a lower-neighborhood
\[
A = U \cap \XX^q = U \cap f^{-1}(-\infty,q],
\]
and an upper-neighborhood
\[
B = U \cap \XX_q = U \cap f^{-1}[q,+\infty).
\]
Then $U = A \cup B$ and $\XX_q^q = A \cap B$, and we desire the exactness of
\[
\tag{$\diamond_{AB}$}
\raisebox{-3.5pc}{
\begin{tikzpicture}[xscale=1.2,yscale=1.2]
%\draw (-2, 0) node(020) {$(\diamond_{AB})$};
\draw (1,1) node(11) {$\Hgr_k(U)$}  ;
\draw (0,0) node(00){$\Hgr_k(A)$} (2,0) node(20){$\Hgr_k(B)$} (1,-1) node(101){$\Hgr_k(\XX_q^q)$} ;
%----
\draw[->] (00) -- (11); 
\draw[->] (20) -- (11);
\draw[<-] (20) -- (101);
\draw[<-] (00) -- (101);
\end{tikzpicture}
}
\]
in whichever dimension~$k$ we are considering. Here are two criteria.

\begin{criterionA}
The levelset $\XX_q^q$ is \emph{ $\Hgr_k$-taut in~$U$} if the map (induced by inclusion)
\[
\alpha_{k+1}\colon
\Hgr_{k+1}(A,\XX_q^q) \to \Hgr_{k+1}(U,B)
\]
is an epimorphism, and the map (induced by inclusion)
\[
\alpha_k\colon \Hgr_k(A,\XX_q^q) \to \Hgr_k(U,B)
\]
is a monomorphism.
\end{criterionA}

\begin{criterionB}
The levelset $\XX_q^q$ is \emph{$\Hgr_k$-taut in~$U$} if the map (induced by inclusion)
\[
\beta_{k+1}\colon
\Hgr_{k+1}(B, \XX_q^q) \to \Hgr_{k+1}(U, A)
\]
is an epimorphism, and the map (induced by inclusion)
\[
\beta_k\colon \Hgr_k(B, \XX_q^q) \to \Hgr_k(U,A)
\]
is a monomorphism.
\end{criterionB}

The maps $\alpha_*, \beta_*$ are excision maps, and they would automatically be isomorphisms if the excision axiom applied to them. For the axiom to apply we would need
\begin{align*}
\operatorname{closure}(B - \XX_q^q) &\subseteq \operatorname{interior}(B)
\\
\operatorname{closure}(A - \XX_q^q) &\subseteq \operatorname{interior}(A)
\end{align*}
for $\alpha_*, \beta_*$ respectively, and this is not true in general.

\begin{proposition}
The two criteria are equivalent.
\end{proposition}
\begin{proof}
We show that the statements for $\alpha_{k+1}, \alpha_k$ together imply the statements for $\beta_{k+1}, \beta_k$ (the converse being symmetric).

The following commutative diagram is obtained by criss-crossing the long exact sequences for the triples $(U, A ,\XX_q^q)$ and $(U, B, \XX_q^q)$:
\[
\begin{tikzpicture}[xscale=1.4,yscale=1.3]
\draw
(-0.1,0) node(00){$\Hgr_{k+1}(B, \XX_q^q)$}
(2,0) node(20){$\Hgr_{k+1}(U, A)$}
(4,0) node(40){$\Hgr_{k}(A, \XX_q^q)$}
(5.9,0) node(60){$\Hgr_{k}(U, B)$}
(1,1) node(11){$\Hgr_{k+1}(U, \XX_q^q)$};
 
 \draw
 (-0.1,2) node(02){$\Hgr_{k+1}(A, \XX_q^q)$}
 (2,2) node(22){$\Hgr_{k+1}(U, B)$}
 (4,2) node(42){$\Hgr_{k}(B, \XX_q^q)$}
 (5.9,2) node(62){$\Hgr_{k}(U, A)$}
 (5,1) node(51){$\Hgr_{k}(U, \XX_q^q)$};
 
\draw[->] (00) to node[below]{$\scriptstyle \beta_{k+1}$} (20);
\draw[->] (00) -- (11);
\draw[<-] (20) -- (11); 
\draw[->] (20) to node[below]{$\scriptstyle \partial$} (40); 
\draw[->] (40) to node[below]{$\scriptstyle \alpha_k$} (60); 

\draw[->] (02) to node[above]{$\scriptstyle \alpha_{k+1}$} (22);
\draw[->] (22) to node[above]{$\scriptstyle \partial$} (42); 
\draw[->] (42) to node[above]{$\scriptstyle \beta_k$} (62); 

\draw[->] (11) -- (22);
\draw[->] (02) -- (11); 
\draw[->] (51) -- (60);
\draw[->] (40) -- (51); 
\draw[->] (51) -- (62);
\draw[->] (42) -- (51); 
\end{tikzpicture}
\]
%
%(The four triangles commute because the maps are induced by inclusion and the hexagon commutes because both paths compose to zero.)
%
Note that $\alpha_{k+1}$ being an epimorphism implies that the upper~$\partial$ is zero, and $\alpha_k$ being a monomorphism implies that the lower~$\partial$ is zero. With that in mind, it becomes a routine diagram-chase to show that $\beta_{k+1}$ is an epimorphism and $\beta_k$ is a monomorphism.
\end{proof}

We use the term \emph{normal neighborhood} to refer to a neighborhood which contains a closed neighborhood. In a normal topological space (such as a compact Hausdorff space), all neighborhoods of a closed set are normal. Closed neighborhoods are trivially normal.

\begin{proposition}
If the levelset $\XX_q^q$ is $\Hgr_k$-taut in some normal neighborhood, then it is $\Hgr_k$-taut in any normal neighborhood.
\end{proposition}

\begin{proof}
Since any two normal neighborhoods contain a closed neighborhood in common, it is enough to show that
\[
\text{$\XX_q^q$ is $\Hgr_k$-taut in~$U$}
\quad
\Leftrightarrow
\quad
\text{$\XX_q^q$ is $\Hgr_k$-taut in~$W$}
\]
whenever $U \subseteq W$ are neighborhoods and $U$ is closed.
Writing $U = A \cup B$ and $W = A' \cup B'$ as usual, we also consider $V = A \cup B'$.

Criterion~A gives the same result for $U$ as for~$V$, by considering
\[
\begin{tikzpicture}[xscale=1.45,yscale=1.3]
\draw (-0.1,0) node(00){$\Hgr_{*}(A, \XX_q^q)$} (1.8,0) node(20){$\Hgr_{*}(A\cup B, A)$}  (4,0) node(40){$\Hgr_{k}(A \cup B', B')$};
 
\draw[->] (00) -- (20);
\draw[->] (20) to node[above]{$\scriptstyle \simeq$} (40); 
\end{tikzpicture}
\]
The right-hand map is an isomorphism by the excision axiom, which applies in this situation because $A \cup B$ is a closed neighborhood of~$A$ in $A \cup B'$.

Criterion~B gives the same result for $V$ as for~$W$, by considering
\[
\begin{tikzpicture}[xscale=1.45,yscale=1.3]
\draw (-0.1,0) node(00){$\Hgr_{*}(B', \XX_q^q)$} (1.8,0) node(20){$\Hgr_{*}(A\cup B', A)$}  (4,0) node(40){$\Hgr_{k}(A' \cup B', A')$};
 
\draw[->] (00) -- (20);
\draw[->] (20) to node[above]{$\scriptstyle \simeq$} (40); 
\end{tikzpicture}
\]
The right-hand map is an isomorphism by excision, since $A \cup B'$ is a closed neighborhood of~$B'$ in $A' \cup B'$.

The result follows.
\end{proof}

\begin{definition}
Accordingly, we say that the levelset $\XX_q^q$ is \emph{$\Hgr_k$-taut} if it is $\Hgr_k$-taut in some, and therefore every, normal neighborhood.
\end{definition}
\begin{definition}\label{taut}
We say that the levelset $\XX_q^q$ is \emph{$\Hgr$-taut} if it is $\Hgr_k$-taut in all dimensions~$k$. This means that for every normal neighborhood~$U$, the maps
\[
\alpha_k: \Hgr_k(A, \XX_q^q) \to \Hgr_k(U,B)
\]
are isomorphisms for all~$k$, or equivalently
\[
\beta_k: \Hgr_k(B, \XX_q^q) \to \Hgr_k(U,A)
\]
are isomorphisms for all~$k$.
\end{definition}

\begin{proposition}
If the levelset $\XX_q^q$ is $\Hgr_k$-taut, then the diagram $(\diamond_{AB})$ is exact for any normal neighborhood $U = A \cup B$.
\end{proposition}

\begin{proof}
Using Criterion~B, say, this is a straightforward chase on the diagram
\[
\begin{tikzpicture}[xscale=1.2,yscale=0.8]
\draw (-0.1,0) node(00){$\Hgr_{k+1}(B, \XX_q^q)$} (2,0) node(20){$\Hgr_{k}(\XX_q^q)$}  (4,0) node(40){$\Hgr_{k}(B)$}
 (5.9,0) node(60){$\Hgr_{k}(B, \XX_q^q)$};
 
 \draw (-0.1,2) node(02){$\Hgr_{k+1}(U, A)$} (2,2) node(22){$\Hgr_{k}(A)$}  (4,2) node(42){$\Hgr_{k}(U)$}
 (5.9,2) node(62){$\Hgr_{k}(U, A)$};
 
\draw[->] (00) -- (20); 
\draw[->] (20) -- (40); 
\draw[->] (40) -- (60); 

\draw[->] (02) --(22);
\draw[->] (22) -- (42); 
\draw[->] (42) -- (62); 

\draw[->] (00) to node[left]{\footnotesize{epi}} (02);
\draw[->] (20) -- (22); 
\draw[->] (40) -- (42);
\draw[->] (60)  to node[right]{\footnotesize{mono}} (62);  
\end{tikzpicture}
\]
for the map of long exact sequences induced by the inclusion $(B, \XX_q^q) \to (U, A)$. 
\end{proof}

This completes our treatment of tautness. Here are some examples.

\begin{proposition}
\label{prop:taut-ex}
The $\RR$-space $\XX = (X, f)$ has $\Hgr$-taut levelsets under any of the following circumstances:

\smallskip
(i)
$X$ is locally compact, $f$ is proper, and $\Hgr$ is Steenrod--Sitnikov homology~\cite{Ferry, Steenrod}.

\smallskip
(ii)
Each $\XX_q^q$ is a deformation retract of some closed neighborhood in $\XX_q$ or $\XX^q$.

\smallskip
(iii)
$X$ is a smooth manifold and $f$ is a proper Morse function.

\smallskip
(iv)
$X$ is a locally compact polyhedron and $f$ is a proper piecewise-linear map.

\smallskip
(v)
$X \subseteq \RR^n \times \RR$ is a closed definable set in some o-minimal structure~\cite{ominimal} and $f$ is the projection onto the second factor. In particular, this applies when $X$ is semialgebraic~\cite{semialg}.

\smallskip
%(vi) {\color{Purple} The complement of a proper $X$ in $\RR^n \times \RR$.}

\end{proposition}
\begin{proof}

\textit{(i)} Steenrod--Sitnikov homology satisfies a strengthened form of the excision axiom~\cite{Steenrod} that does not require any restriction on the subspaces under consideration. Therefore maps in Definition~\ref{taut} are isomorphisms for any levelset $\XX_q^q$.

\smallskip
\textit{(ii)} Let $C_1$ be a closed neighborhood of $\XX_q^q$. We know $\XX_q^q$ is a deformation retract of a closed neighborhood $C_2$ in $\XX^q$. We may assume without loss of generality that $C_2 \subseteq C_1$. Let $C = C_2 \cup (C_1 \cap \XX_q)$. %Take $A = C \cap \XX^q = C_2$ and $B =C \cap \XX_q$. 
The homology groups $\Hgr_k(C_2, \XX_q^q)$ and $\Hgr_k(C,C\cap \XX_q)$ are trivial for every $k$ and therefore isomorphic, implying that $\XX_q^q$ is $\Hgr$-taut.

\smallskip
\textit{(iii)}, \textit{(iv)} and \textit{(v)} follow from \textit{(ii)}.  In particular, we prove \textit{(v)} by applying \cite[Corollary 3.9, Chapter 8]{ominimal}.
\end{proof}

\begin{remark}
We occasionally need to consider Mayer--Vietoris diamonds in relative homology. We establish their exactness individually as they occur.
\end{remark}

%-----------------------------------------------------------------
\subsection{Additivity}
\label{sec:additivity}

We are now ready to prove that the four measures $\mu^\xxs_{\Hgr\XX}$ are additive.

\begin{theorem}
\label{thm:additivity}
Let $\Hgr$ be a homology functor with field coefficients and let $\XX = (X,f)$ be an $\RR$-space whose levelsets are $\Hgr$-taut. Then $\muud_{\Hgr\XX}$, $\mudd_{\Hgr\XX}$, $\muuu_{\Hgr\XX}$, and $\mudu_{\Hgr\XX}$ are additive.
\end{theorem}

\begin{proof}
Let $R = [a,b] \times [c,d]$ and consider a horizontal split
\[
R_1 = [a,p] \times [c,d],
\qquad
R_2 = [p,b] \times [c,d],
\]
so $a < p < b < c < d$.
The diagram
 \[
% \XX_{\{a, p, b; c, d \}}:  \raisebox{-4.25ex}{
\begin{tikzpicture}[xscale=1.2,yscale=1.2]
\draw (1,1) node(11) {$\XX_b^c$} ;
\draw (0,0) node(00){$\XX_b^b$};
\draw (-1,1) node(011){$\XX_p^b$} ;
\draw (2,0) node(20) {$\XX_c^c$} ;
\draw (4,0) node(40) {$\XX_d^d$} ;
\draw (3,1) node(31) {$\XX_c^d$} ;
\draw (-2,0) node(020) {$\XX_p^p$} ;
\draw (-2,2) node(022) {$\XX_a^b$} ;
\draw (-4,0) node(040) {$\XX_a^a$} ;
\draw (-3,1) node(031) {$\XX_a^p$} ;
\draw (0,2) node(02) {$\XX_p^c$} ;
%----
\draw[->] (00) -- (011); 
\draw[->] (20) -- (11);
\draw[->] (00) -- (11);
\draw[->] (040) -- (031);
\draw[->] (020) -- (031);
\draw[->] (020) -- (011);
\draw[->] (20) -- (31);
\draw[->] (40) -- (31);
\draw[->] (11) -- (02);
\draw[->] (011) -- (02);
\draw[->] (031) -- (022);
\draw[->] (011) -- (022);
\end{tikzpicture}
%}
%
\]
contains the zigzags $\XX_{\{a,b,c,d\}}$, $\XX_{\{a,p,c,d\}}$, $\XX_{\{p,b,c,d\}}$ for all three rectangles. When we apply $\Hgr$, the two diamonds in the resulting diagram are exact since the levelsets $\XX_p^p, \XX_b^b$ are $\Hgr$-taut.
We calculate:
\[
\begin{array}{ll}
	\mudu_{\Hgr\XX}(R)
	= \zzhsplituuu{1}
	& = \zzhsplituuu{3} + \zzhsplituuu{2} \\
	& = \zzhsplituuu{6} + \zzhsplituuu{4} \\
	& =  \zzhsplituuu{7} + \zzhsplituuu{5}
	= \mudu_{\Hgr\XX}(R_1) + \mudu_{\Hgr\XX}(R_2).
\end{array}
\]
In the first line we add two extra nodes to refine the 7-term zigzag to a 9-term zigzag and use the Restriction Principle. In the second line we use the Diamond Principle twice. In the third line we drop two nodes in each term and use the Restriction Principle again.

Similar calculations establish the additivity of $\muud_{\Hgr\XX}$, $\mudd_{\Hgr\XX}$ and $\muuu_{\Hgr\XX}$ under horizontal splitting.
%
%\bigskip
%\bigskip
%{\color{WildStrawberry}
%Perhaps we can do all four cases of the horizontal split?
%}
%\bigskip
%\bigskip
%
Additivity under vertical splitting follows by coordinate-reversal symmetry.
\end{proof}

\subsection{Finiteness}
\label{sec:finiteness}

We now consider the finiteness of the four r-measures $\mu^\xxs_{\Hgr\XX}$. As discussed in Section~\ref{subsec:measures}, finiteness of an r-measure implies that its decorated persistence diagram is defined everywhere in~$\Hp$; in general the diagram is defined in the finite support of the r-measure.

It turns out to be essentially the same issue as the finiteness of the \emph{well groups}~\cite{bendich,EMP}. Well groups measure that part of the homology of a fiber $\Hgr(\XX_m^m)$ of an $\RR$-space that is stable under $\epsilon$-perturbations of the coordinate. One defines
\[
\Wgr (\Hgr\XX; m, \epsilon)
=
\bigcap_g
\operatorname{image} \big[
\Hgr (g^{-1}(q)) \longrightarrow \Hgr \XX_{q-\epsilon}^{q+\epsilon}
\big]
\]
where the intersection is taken over all $\epsilon$-perturbations $g$ of the coordinate~$f$, perhaps in a suitable regularity class. Considering the perturbations $g = f\pm\epsilon$, it follows that the well group is contained in%
\footnote{%
Indeed, the well group is equal to this intersection if the class of perturbations has $\Hgr$-taut fibers.
}
\[
\;\operatorname{image}\!
\big[
\Hgr \XX_{q-\epsilon}^{q-\epsilon}
\longrightarrow \Hgr \XX_{q-\epsilon}^{q+\epsilon}
\big]
\,\cap\,
\operatorname{image}\!
\big[
\Hgr \XX_{q+\epsilon}^{q+\epsilon}
\longrightarrow \Hgr \XX_{q-\epsilon}^{q+\epsilon}
\big]
\]
and therefore its rank is bounded by
\[
\langle
\onof{6} \mid
\Hgr_{q-\epsilon}^{q-\epsilon}
\longrightarrow
\Hgr_{q-\epsilon}^{q+\epsilon}
\longleftarrow
\Hgr_{q+\epsilon}^{q+\epsilon}
\rangle
=
\langle
\onof{6} \mid \Hgr \XX_{\{q-\epsilon, q+\epsilon\}}
\rangle.
\]
This takes the same form as the term that we need to bound.

\begin{lemma}
\label{bendichprop}
Let $\XX = (X, f)$ be an $\RR$-space and $\Hgr$ be a homology functor. For any rectangle $R= [a, b]\times [c, d]$ with $a<b<c<d$ we have
\[
 \muud_{\Hgr\XX}(R)
 + \mudd_{\Hgr\XX}(R)
 + \muuu_{\Hgr\XX} (R)
 + \mudu_{\Hgr\XX} (R)
 \leq
 \langle  \wud \mid \Hgr \XX_{\{a, b, c, d\}}\rangle
 =
 \langle  \onof{6} \mid \Hgr \XX_{\{b, c\}}\rangle.
 \]
\end{lemma}

\begin{proof}
%Each of the four types $\zzud$, $\zzdd$, $\zzuu$, $\zzdu$ restricts to $\wud$  over the middle three nodes.
%
By the Restriction Principle % it follows that
\begin{align*}
\langle  \wud  \rangle
&\geq
\langle \zzud\rangle + \langle \zzdd\rangle
+ \langle \zzuu\rangle + \langle \zzdu\rangle 
\\
&=
(\muud_{\Hgr\XX} + \mudd_{\Hgr\XX}
+ \muuu_{\Hgr\XX} + \mudu_{\Hgr\XX})(R) .  
\qedhere
\end{align*}
\end{proof}

\begin{proposition}
\label{prop:finite1}
Let $\XX=(X, f)$. Then $\mudu_{\Hgr \XX}$, $\muuu_{\Hgr \XX}$, $\muud_{\Hgr \XX}$ and $\mudd_{\Hgr \XX}$ are finite for any~$\Hgr$ under any of the following circumstances:

\smallskip
(i)
$X$ is a locally compact polyhedron and $f$ a proper continuous map.

\smallskip
(iii)
$X$ is a smooth manifold and $f$ is a proper Morse function.

\smallskip
(iv)
$X$ is a locally compact polyhedron and $f$ is a proper piecewise-linear map.

\smallskip
(v)
$X \subseteq \RR^n \times \RR$ is a closed definable set in some o-minimal structure and $f$ is the projection onto the second factor.
\end{proposition}

\begin{proof}
In cases (iii), (iv), (iv) each slice $\XX_b^c$ has the homotopy type of a finite cell complex, and therefore has finite-dimensional homology.

The proof of (i) is a little more involved.
%\end{proof}
%
%With a little more work, we get the following result. We will typically use this with Steenrod--Sitnikov homology, where the tautness of the fibers is not in question.
%
Let $R = [a, b] \times [c, d]$. Choose $m$ and $\epsilon>0$ such that $b+2\epsilon < m < c-2\epsilon$, and approximate $f$ with a piecewise-linear map $g \colon X \to \RR$ for which $\|g - f\|\leq \epsilon$. Then $g$ is also proper, and $Y = g^{-1}(m)$ is triangulable as a finite simplicial complex and is $\Hgr$-taut as a fiber of $(X,g)$.

We can split the neighborhood $\XX_b^c$ into lower- and upper-neighborhoods of~$Y$ by defining
\[
U = \XX_b^c \cap g^{-1}(-\infty,m],
\quad
V = \XX_b^c \cap g^{-1}[m,+\infty).
\]
Thus $\XX_b^c = U \cup V$ and $Y = U \cap V$.
Since $||g-f||\leq \epsilon$, we also have $\XX_{b}^{b} \subseteq U$ and $\XX_{c}^{c} \subseteq V$.

Consider the following diagram of spaces and maps:
\[
\begin{tikzpicture}[xscale=1.2,yscale=1.2, font=\small]
\draw
(1,1) node(11) {$\Hgr(\XX_{b}^{c})$} ;
\draw
(0,0) node(00){$\Hgr(U)$}
(2,0) node(20){$\Hgr(V)$}
(1,-1) node(101){$\Hgr(Y)$}
(-1,-1) node(001){$\Hgr(\XX_{b}^{b})$}
(-3,-1) node(0301){$\Hgr(\XX_{a}^{a})$}
(-2,0) node(020){$\Hgr(\XX_{a}^{b})$}
(3,-1) node(201){$\Hgr(\XX_{c}^{c})$}
(5,-1) node(501){$\Hgr(\XX_{d}^{d})$}
(4,0) node(40){$\Hgr(\XX_{c}^{d})$};

\draw[->] (0301) -- (020); 
\draw[->] (001) -- (020);
\draw[->] (001) -- (00); 
\draw[->] (201) -- (20);
\draw[->] (00) -- (11); 
\draw[->] (20) -- (11);
\draw[<-] (20) -- (101);
\draw[<-] (00) -- (101);

\draw[->] (501) -- (40); 
\draw[->] (201) -- (40);
\end{tikzpicture}
\]
By the Restriction and Diamond Principles (since $Y$ is $\Hgr$-taut) we have
\[
\left\langle \zzspider{1} \right\rangle
=\left \langle \zzspider{2}\right\rangle
= \left\langle \zzspider{3} \right\rangle
\leq \Hgr(Y) < \infty.
\]
The result now follows from Lemma~\ref{bendichprop}.
\end{proof}

%-----------------------------------------------------------------
\subsection{The Four Diagrams of Parametrized Homology}
\label{paramhom}

Let $\XX = (X, f)$ be an $\RR$-space and let $\Hgr$ be a homology functor with field coefficients. 
Quantities $\mudd_\XX$, $\mudu_\XX$, $\muud_\XX$, and $\muuu_\XX$ capture the way topological features of $\XX$ perish at endpoints. When they are r-measures, each defines a persistence diagram via the Equivalence Theorem. We denote these four decorated persistence diagrams by $\Dgm^\dds(\XX)$, $\Dgm^\dus(\XX)$, $\Dgm^\uds(\XX)$, and $\Dgm^\uus(\XX)$. These, collectively, comprise the \emph{parametrized homology} of $\XX$ with respect to the homology functor $\Hgr$. 

\begin{theorem}\label{definepar}
We can define parametrized homology of ${\XX = (X, f)}$ when: 

\smallskip
(i) $X$ is a locally compact polyhedron, $f$ is proper, and $\Hgr$ is Steenrod--Sitnikov homology.

\smallskip
(iii)
$X$ is a smooth manifold and $f$ is a proper Morse function.

\smallskip
(iv)
$X$ is a locally compact polyhedron and $f$ is a proper piecewise-linear map.

\smallskip
(v)
$X \subseteq \RR^n \times \RR$ is a closed definable set in some o-minimal structure and $f$ is the projection onto the second factor.
\end{theorem}

\begin{proof}
Additivity follows from Proposition~\ref{prop:taut-ex} and finiteness from Proposition~\ref{prop:finite1}.
\end{proof}

%--------------------------------------------------------------
\subsection{Levelset Zigzag Persistence}
\label{subsec:LZZ}

In some situations finite zigzag diagrams carry all the needed information. Let $\XX = (X, f)$ be an $\RR$-space constructed as follows.  There is a finite set of real-valued indices $S = \{a_1, . . . , a_n\}$ (listed in increasing order), called the \emph{critical values} of $\XX$. Then:
\begin{itemize}
\item
For $1 \leq i \leq n$, $V_i$ is a locally path-connected compact space;
\item
For $1 \leq i \leq n-1$, $E_i$ is a locally path-connected compact space;
\item
For $1 \leq i \leq n-1$, $l_i\colon E_i \to V_i$ and $r_i\colon E_i \to V_{i+1}$ are continuous maps.
\end{itemize}
Let $X$ be the quotient space obtained from the disjoint union of the spaces $V_i \times \{a_i\}$ and $E_i \times [a_i,a_{i+1}]$ by making the identifications $(l_i(x),a_i) \sim (x,a_i)$ and $(r_i(x), a_{i+1})\sim (x,a_{i+1})$ for all $i$ and all $x \in E_i$. Let $f\colon X \to \RR$ be the projection onto the second factor. In this paper, we follow Carlsson et al.\ \cite{Zigzagpersistenthomologyandreal} in calling such an $\XX = (X,f)$ a \emph{Morse type} $\RR$-space. (In \cite{Silva_Munch_Patel_2015} they are called \emph{constructible $\RR$-spaces}.) Such $\RR$-spaces include $\XX=(X, f)$, where $X$ is a compact manifold and $f$ a Morse function, and $X$ a compact polyhedron and $f$ piecewise linear. 

We can track the appearance and disappearance of topological features using \emph{levelset zigzag persistence} construction~\cite{Zigzagpersistenthomologyandreal}. Given $\XX=(X, f )$ of Morse type, select a set of indices $s_i$ which satisfy
\[
-\infty < s_0 < a_1 \ldots < a_{n} < s_n < \infty, 
\]
and build a zigzag diagram that serves as a model for~$\XX$:
\begin{center}
\begin{tikzpicture}[xscale=1.05,yscale=1.05]
\draw (-2, 0.5)node(010){$\XX_{\{s_0, \ldots, s_n\}}:$} (1,1) node(11) {$\XX_{s_0}^{s_1}$} (3,1) node(31) {$\XX_{s_1}^{s_2}$} (5,1) node(51) {$\ldots$} (7,1) node(71) {$\XX_{s_{n-2}}^{s_{n-1}}$} (9,1) node(91) {$\XX_{s_{n-1}}^{s_n}$} ;
\draw (0,0) node(00){$\XX_{s_0}^{s_0}$} (2,0) node(20){$\XX_{s_1}^{s_1}$}  (4,0) node(40){$\XX_{s_2}^{s_2}$} (6,0) node(60){$\XX_{s_{n-2}}^{s_{n-2}}$}  (8,0) node(80){$\XX_{s_{n-1}}^{s_{n-1}}$} (10,0) node(100){$\XX_{s_{n}}^{s_{n}}.$}  ;

\draw[->] (00) -- (11); 
\draw[->] (20) -- (11);
\draw[->] (20) -- (31);
\draw[->] (40) -- (31);
\draw[->] (40) -- (51);
\draw[->] (60) -- (51);  
\draw[->] (80) -- (71);
\draw[->] (60) -- (71);  
\draw[->] (80) -- (91);
\draw[->] (100) -- (91);  
\end{tikzpicture}
\end{center}
Apply homology functor $\Hgr$ to obtain:
\begin{center}
\begin{tikzpicture}[xscale=1.05,yscale=1.05]
\draw (-2, 0.5)node(010){$\Hgr\XX_{\{s_0, \ldots, s_n\}}:$} (1,1) node(11) {$\Hgr(\XX_{s_0}^{s_1})$} (3,1) node(31) {$\Hgr(\XX_{s_1}^{s_2})$} (5,1) node(51) {$\ldots$} (7,1) node(71) {$\Hgr(\XX_{s_{n-2}}^{s_{n-1}})$} (9,1) node(91) {$\Hgr(\XX_{s_{n-1}}^{s_n})$} ;
\draw (0,0) node(00){$\Hgr(\XX_{s_0}^{s_0})$} (2,0) node(20){$\Hgr(\XX_{s_1}^{s_1})$}  (4,0) node(40){$\Hgr(\XX_{s_2}^{s_2})$} (6,0) node(60){$\Hgr(\XX_{s_{n-2}}^{s_{n-2}})$}  (8,0) node(80){$\Hgr(\XX_{s_{n-1}}^{s_{n-1}})$} (10,0) node(100){$\Hgr(\XX_{s_{n}}^{s_{n}}).$}  ;

\draw[->] (00) -- (11); 
\draw[->] (20) -- (11);
\draw[->] (20) -- (31);
\draw[->] (40) -- (31);
\draw[->] (40) -- (51);
\draw[->] (60) -- (51);  
\draw[->] (80) -- (71);
\draw[->] (60) -- (71);  
\draw[->] (80) -- (91);
\draw[->] (100) -- (91);  
\end{tikzpicture}
\end{center}
This quiver representation is decomposable by Gabriel's Theorem~\cite{quiver}. 

We translate between the notation of intervals that appear in the levelset zigzag persistence of $\XX$  and critical values as follows:

\begin{alignat*}{4}
[\Hgr(\XX_{s_{i-1}}^{s_i}), \Hgr(\XX_{s_{j-1}}^{s_j})] & \text{ corresponds to } & [a_i, a_j] & \text{ for }1\leq i \leq j \leq n,\\
 [\Hgr(\XX_{s_{i-1}}^{s_i}), \Hgr(\XX_{s_{j-1}}^{s_j})]  & \text{ corresponds to }&[a_i, a_j) & \text{ for }1\leq i < j \leq n+1,\\
 [\Hgr(\XX_{s_i}^{s_i}), \Hgr(\XX_{s_{j-1}}^{s_j})]  &\text{ corresponds to } & (a_i, a_j] & \textrm{ for }1\leq i \leq j \leq n,\\
[\Hgr(\XX_{s_i}^{s_i}), \Hgr(\XX_{s_{j-1}}^{s_{j-1}})]  &\text{ corresponds to } & (a_i, a_j) & \textrm{ for }1\leq i < j \leq n+1.\\
\end{alignat*}
We interpret $a_0$ as $-\infty$ and $a_{n+1}$ as $\infty$.

The collection of these pairs of critical values, taken with multiplicity and labelled by the interval type is called the levelset zigzag persistence diagram of $\XX$ and denoted by $\DgmZZ(\Hgr\XX)$.

The four quantities defined in Section~\ref{subsec:4-measures}, $\muud_{\Hgr\XX}$, $\mudd_{\Hgr\XX}$, $\muuu_{\Hgr\XX}$ and $\mudu_{\Hgr\XX}$, are measures when $\XX$ is of Morse type. Additivity follows from (ii) of Proposition~\ref{prop:taut-ex}, while finiteness from the assumption that all interlevelsets and levelsets have finite dimensional homology groups. 

In fact, parametrized homology and levelset zigzag persistence of a Morse type $\RR$-space carry the same information, as the following theorem demonstrates. 
\begin{theorem}\label{16to4Morse}
If $\XX$ is an $\RR$-space of Morse type with critical values
\[
a_1 < a_2 < \ldots < a_n,
\]
then the levelset zigzag persistence diagram of $\XX$, $\DgmZZ(\Hgr\XX)$, contains the same information as the four diagrams  $\Dgm^\dds(\XX)$, $\Dgm^\dus(\XX)$, $\Dgm^\uds(\XX)$, and $\Dgm^\uus(\XX)$.  To be more precise, 
\begin{alignat*}{3}
(a_i, a_j) \in \Dgm^\uds\!(\Hgr\XX) & \;\;\text{if and only if}\;\;
& (a_i^+, a_j^-)\in \DgmZZ(\Hgr\XX) 
\\
[a_i, a_j) \in \Dgm^\dds\!(\Hgr\XX) & \;\;\text{if and only if}\;\;
& (a_i^-, a_j^-)\in  \DgmZZ(\Hgr\XX) 
\\
(a_i, a_j]\in \Dgm^\uus\!(\Hgr\XX) & \;\;\text{if and only if}\;\;
& (a_i^+, a_j^+)\in  \DgmZZ(\Hgr\XX)  
\\
[a_i, a_j]\in \Dgm^\dus\!(\Hgr\XX) & \;\;\text{if and only if}\;\;
& (a_i^-, a_j^+) \in  \DgmZZ(\Hgr\XX) . 
\end{alignat*}
Diagrams $\Dgm^\uds\!(\Hgr\XX)$, $\Dgm^\dds\!(\Hgr\XX)$, $\Dgm^\uus\!(\Hgr\XX)$ and $\Dgm^\dus\!(\Hgr\XX)$ contain no decorated points with nonzero multiplicity other than those specified above.
\end{theorem}
\begin{proof}
First we prove that if $[a_i, a_j]$ with multiplicity $m$, $m\geq 1$, is contained in the levelset zigzag persistence diagram of $\XX$, then $\textrm{m}_{\Dgm^\dus(\XX)}(a_i^-, a_j^+)= m$.

%\begin{center}
%\begin{tikzpicture}[xscale=1.1,yscale=1.1]
%\draw (1,1) node(11) {$\Hgr(\XX_{s_i}^{s_{i+1}})$} (3,1) node(31) {$\Hgr(\XX_{s_{i+1}}^{s_{i+2}})$} (5,1) node(51){$\ldots$} (7,1) node(71) {$\Hgr(\XX_{s_{j-2}}^{s_{j-1}})$} (9,1) node(91) {$\Hgr(\XX_{s_{j-1}}^{s_j})$} ;
%\draw (0,0) node(00){$\Hgr(\XX_{s_i}^{s_i})$} (2,0) node(20){$\Hgr(\XX_{s_{i+1}}^{s_{i+1}})$}  (4,0) node(40){$\Hgr(\XX_{s_{i+2}}5^{s_{i+2}})$} (6,0) node(60){$\Hgr(\XX_{s_{j-2}}^{s_{j-2}})$}  (8,0) node(80){$\Hgr(\XX_{s_{j-1}}^{s_{j-1}})$} (10,0) node(100){$\Hgr(\XX_{s_{j}}^{s_{j}}).$}  ;

%\draw[->] (00) -- (11); 
%\draw[->] (20) -- (11);
%\draw[->] (20) -- (31);
%\draw[->] (40) -- (31);
%\draw[->] (40) -- (51);
%\draw[->] (60) -- (51);  
%\draw[->] (80) -- (71);
%\draw[->] (60) -- (71);  
%\draw[->] (80) -- (91);
%\draw[->] (100) -- (91);  
%\end{tikzpicture}
%\end{center}
We select a set of indices $s_i$ which satisfy
\[
-\infty < s_0 < a_1 < s_1 < a_2 < \ldots  < s_{n-1} < a_n < s_n < \infty.
\]

By definition $[a_i, a_j]$ appears in the levelset zigzag persistence diagram with multiplicity $m$ if and only if 
\[
\langle [\Hgr(\XX_{s_{i-1}}^{s_i} ), \Hgr(\XX_{s_{j-1}}^{s_j})] \mid \Hgr\XX_{\{s_0, \ldots, s_n\}} \rangle=m.
\]
By the Diamond and the Restriction Principle
\[
\begin{array}{lcl}
\langle [\Hgr(\XX_{s_{i-1}}^{s_i} ), \Hgr(\XX_{s_{j-1}}^{s_j})] \mid \Hgr\XX_{\{s_0, \ldots, s_n\}} \rangle &= &\langle [\Hgr(\XX_{s_{i-1}}^{s_i} ), \Hgr(\XX_{s_{j-1}}^{s_j})] \mid \Hgr\XX_{\{s_{i-1}, s_i, s_{j-1}, s_j\}} \rangle.  \\
\end{array}
\]
Choose $\epsilon < \frac{1}{2}\min \{ a_i-s_{i-1}, s_j-a_{j}\}$. Observe the diagram below.
\[
\begin{tikzpicture}[xscale=1,yscale=1.2]
\draw (1,1) node(11) {$\Hgr(\XX_{s_{i-1}}^{a_i-\epsilon})$} 
(3,1) node(31) {$\Hgr(\XX_{a_i}^{a_i-\epsilon})$}
(5,1) node(51) {$\Hgr(\XX_{a_i}^{s_{i}})$} ;
\draw (4,2) node(42) {$\Hgr(\XX_{a_i-\epsilon}^{s_{i}})$} ;
\draw (2,2) node(22){$\Hgr(\XX_{s_{i-1}}^{a_{i}})$}; 
\draw (3,3) node(33){$\Hgr(\XX_{s_{i-1}}^{s_{i}})$}; 

\draw (0,0) node(00){$\Hgr(\XX_{s_{i-1}}^{s_{i-1}})$} (2,0) node(20){$\Hgr(\XX_{a_i-\epsilon}^{a_i-\epsilon})$} (4,0) node(40){$\Hgr(\XX_{a_i}^{a_i})$} (6,0) node(60){$\Hgr(\XX_{s_{i}}^{s_{i}})$}  (6,2) node(62){$\Hgr(\XX_{a_{i}}^{s_{j-1}})$} (8,0) node(80){$\Hgr(\XX_{s_{j-1}}^{s_{j-1}})$}  (10,0) node(100){$\Hgr(\XX_{a_{j}}^{a_{j}})$} (12,0) node(120){$\Hgr(\XX_{a_{j}+\epsilon}^{a_{j}+\epsilon})$} (14,0) node(140){$\Hgr(\XX_{s_{j}}^{s_{j}})$}
 (7,3) node(73){$\Hgr(\XX_{a_{i}}^{a_{j}})$} (9,1) node(91){$\Hgr(\XX_{s_{j-1}}^{a_{j}})$} 
 (11,1) node(111){$\Hgr(\XX_{a_{j}}^{a_{j+\epsilon}})$}
 (13,1) node(131){$\Hgr(\XX_{a_{j}+\epsilon}^{s_{j}})$}
 (12,2) node(122){$\Hgr(\XX_{a_{j}}^{s_{j}})$}
 (11,3) node(113){$\Hgr(\XX_{s_{j-1}}^{s_{j}})$}
 (10,2) node(102){$\Hgr(\XX_{s_{j-1}}^{a_{j}+\epsilon})$}

 (7,1) node(71){$\Hgr(\XX_{s_{i}}^{s_{j-1}})$} (8,2) node(82){$\Hgr(\XX^{a_j}_{s_{i}})$}  ;
\draw[->] (91) -- (102); 
\draw[->] (102) -- (113); 
\draw[->] (122) -- (113); 
\draw[->] (111) -- (122); 
\draw[->] (111) -- (102); 
\draw[->] (131) -- (122); 
\draw[->] (140) -- (131); 
\draw[->] (120) -- (131); 
\draw[->] (120) -- (111); 
\draw[->] (100) -- (111); 
\draw[->] (91) -- (82); 
\draw[->] (100) -- (91); 
\draw[->] (80) -- (91); 
\draw[->] (82) -- (73); 
\draw[->] (62) -- (73); 
\draw[->] (42) -- (33); 
\draw[->] (22) -- (33); 
\draw[->] (31) -- (22); 
\draw[->] (31) -- (42); 
\draw[->] (11) -- (22); 
\draw[->] (00) -- (11); 
\draw[->] (20) -- (11);
\draw[->] (20) -- (31);
\draw[->] (40) -- (31);
\draw[->] (60) -- (71);  
\draw[->] (80) -- (71);
\draw[->] (60) -- (51);
\draw[->] (51) -- (42);
\draw[->] (40) -- (51);

\draw[->] (51) -- (62);
\draw[->] (71) -- (62);
\draw[->] (71) -- (82);

\end{tikzpicture}
\]
Using the Diamond Principle and the Restriction Principle we calculate:
\[
\begin{array}{ll}
\langle [\Hgr(\XX_{s_{i-1}}^{s_i} ), \Hgr(\XX_{s_{j-1}}^{s_j})] \mid \Hgr\XX_{\{s_{i-1}, s_i, s_{j-1}, s_j\}} \rangle
	&= \morse{1}\\
	& = \morse{2}  \\
	& = \morse{3}  \\
	& = \morse{4}  \\
	&= \morse{5} \\
	&\\
	&= \mudu_{\Hgr\XX} ([a_i-\epsilon, a_i]\times [a_j, a_j+\epsilon]).
\end{array}
\]
 In the second line we used the fact that $\XX$ is of Morse type. This implies $\XX_{s_{i-1}}^{s_{i-1}}$ is homotopy equivalent to $\XX_{s_{i-1}}^{a_{i}-\epsilon}$, $\XX_{s_{i-1}}^{a_{i}}$ to $\XX_{s_{i-1}}^{s_{i}}$,  $\XX_{a_{j}+\epsilon}^{s_{j}}$ to $\XX_{s_{j}}^{s_{j}}$ and $\XX_{a_{j}}^{s_{j}}$ to $\XX_{s_{j-1}}^{s_{j}}$ for all sufficiently small $\epsilon$. Therefore
\[
\textrm{m}_{\Dgm^\dus(\XX)}(a_i^-, a_j^+)= \lim_{\epsilon\to 0} \mudu_{\Hgr\XX} ([a_i-\epsilon, a_i]\times [a_j, a_j+\epsilon]) = m.
\]

We must now show that $\Dgm^\dus(\XX)$ contains only points of the type $(a_i^-, a_j^+)$, where $a_i$ and $a_j$ are critical values of $\XX$. For any $p\in \RR$, an $\epsilon>0$ exists such that $\XX_p^{p+\epsilon}$ and $\XX_{p-\epsilon}^p$ strongly deformation retracts to $\XX_p^p$. This means that $\Hgr(\XX_p^{p+\epsilon}) \cong \Hgr(\XX_p^p) \cong  \Hgr(\XX_{p-\epsilon}^p)$, forcing
\[
\left\langle \critical{1}\, |\, \Hgr\XX_{\{p-\epsilon, p\}} \right\rangle  =
\left\langle \critical{4}\, |\, \Hgr\XX_{\{p, p+\epsilon\}} \right\rangle = 0.
\]
For $\epsilon$ small enough
\[
\zzud, \zzdd, \zzuu
\]
all appear with 0 multiplicity for any $p$ and $q$ in the quiver decomposition of $\Hgr\XX_{\{p-\epsilon, p, q, q+\epsilon\}}$. This holds since by the restriction principle
\[
0 \leq \left\langle \zzud\, |\, \Hgr\XX_{\{p-\epsilon, p, q, q+\epsilon\}}) \right\rangle, \left\langle \zzuu\, |\, \Hgr\XX_{\{p-\epsilon, p, q, q+\epsilon\}}) \right\rangle \leq \left\langle \critical{1}\, |\, \Hgr\XX_{\{p-\epsilon, p\}} \right\rangle=0
\]
and 
\[
0 \leq \left\langle \zzdd\, |\, \Hgr\XX_{\{p-\epsilon, p, q, q+\epsilon\}}) \right\rangle \leq \left\langle \critical{4}\, |\, \Hgr\XX_{\{q, q+\epsilon\}} \right\rangle =0.
\]
So $\Dgm^\dus(\XX)$ contains exactly points that correspond to intervals of type $[a_i, a_j]$ in $\DgmZZ(\Hgr\XX)$.

We prove the statement for other measures similarly.
\end{proof}
\bigskip

%-----------------------------------------------------------------
\subsection{Sixteen behaviors}
\label{sec:16}

%\begin{example}
%Let $\XX = (X, f)$ be as in Example~\ref{levelset}. A 0-homology cycle $c_\XX$ depicted in Figure~\ref{genhom} in $\XX$ ceases to exist beyond $a_1$ and $a_5$.
%\begin{figure}[h!]
%\begin{center}
%\includegraphics[scale=0.45]{hlacecikel.pdf}
%\end{center}
%\caption{\label{genhom}0-homology cycle $c_\XX$ ceases to exist beyond $a_1$ and $a_5$.}
%\end{figure}
%\end{example}

Let $\XX$ be an $\RR$-space. Depending on the way a feature perishes and whether the corresponding interval is closed or open at endpoints, there are sixteen different cases that can occur (see ~Figure~\ref{fig:16}). 
\begin{figure}[h!]
\begin{center}
\includegraphics[scale = 1.1]{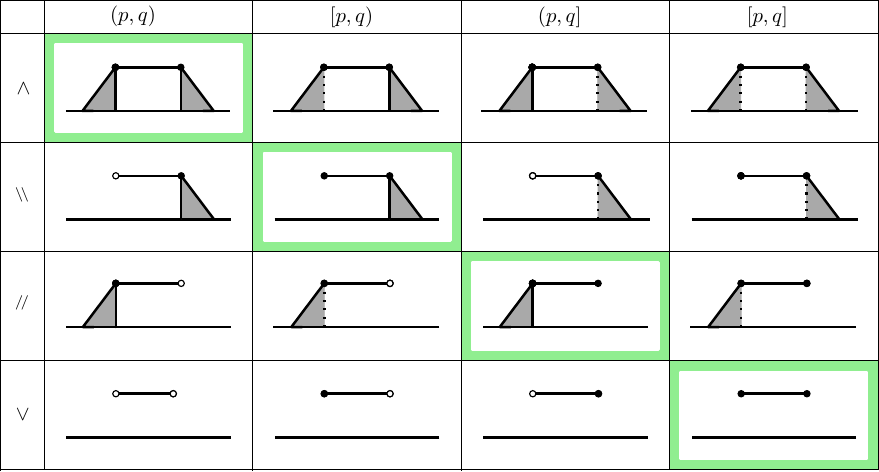}
\caption{Different ways of dying at endpoints.}
\label{fig:16}
\end{center}
\end{figure}
For a Morse type $\RR$-space ${\XX =(X, f)}$, where $X$ is compact, this number drops down to four (highlighted green in Figure~\ref{fig:16}) as demonstrated by Theorem~\ref{16to4Morse}. Something similar occurs when $X$ is a locally compact polyhedron, $f$ a proper continuous map and $\Hgr$ the Steenrod--Sitnikov homology functor. 

The following theorem, inspired by Frosini et al.~\cite{CFFFL}, relies heavily on the continuity property of \v{C}ech homology~\cite{foundations}. For a wide variety of coefficient groups (infinitely divisible; finite exponent)~\cite{Steenrod} \v{C}ech homology coincides with Steenrod--Sitnikov homology. In particular, this is the case for some of the more common fields we may be interested in: $\ff_p$, $\qq$, $\rr$.

\begin{theorem}\label{frosinithm}
Let $\XX =(X, f)$.  We assume that $X$ is a locally compact polyhedron, $f$ is a proper continuous map, and $\Hgr$ is the Steenrod--Sitnikov homology functor with coefficients in $\ff_p$, $\qq$ or $\rr$. 
Then:
\begin{alignat*}{4}
&
\Dgm^\uds\!(\Hgr\XX) && \;\;\text{contains only points of type}
&&
\raisebox{-1.4ex}{
\includegraphics[scale=1.25]{plusminus.pdf}}
\!= (p^+,q^-)
&& = (p,q)
\\
&
\Dgm^\dds\!(\Hgr\XX) && \;\;\text{contains only points of type}
&&
\raisebox{-1.4ex}{
\includegraphics[scale=1.25]{minusminus.pdf}}
\!= (p^-,q^-)
&& = [p,q)
\\
&
\Dgm^\uus\!(\Hgr\XX) && \;\;\text{contains only points of type}
&&
\raisebox{-1.4ex}{
\includegraphics[scale=1.25]{plusplus.pdf}}
\!= (p^+,q^+)
&& = (p,q]
\\
&
\Dgm^\dus\!(\Hgr\XX) && \;\;\text{contains only points of type}
&&
\raisebox{-1.4ex}{
\includegraphics[scale=1.25]{minusplus.pdf}}
\!= (p^-,q^+)
&& = [p,q]
\end{alignat*}
In other words, the four possible decorations correspond exactly to the four ways in which a feature can perish at the ends of its interval.
 \end{theorem}
 
Let $a<b<m<c<d$. We fix a piecewise-linear structure on $X$, and approximate $f \colon X \to \RR$ with a piecewise-linear map $g \colon X \to \RR$ for which $||g - f||\leq \min \{\frac{c-m}{2}, \frac{m-b}{2}\}$.  The preimage $Y=g^{-1}(m)$ is a finite simplicial complex. Let 
\[
{V_q = g^{-1}((\infty, m])} \cap \XX_q\,\quad\textrm{and}\quad  {U^q = g^{-1}([m, \infty))} \cap \XX^q\, \textrm{ for }q\in \RR.
\]
%
%{\color{YellowGreen} Let's make $Y$ a little thicker so that there is no doubt about MV?}
%
In the proof of Theorem~\ref{frosinithm} we will make use of diagrams of this type:
\[
\begin{tikzpicture}[xscale=1.2,yscale=1.2]
\draw (-2,2) node(022) {$\Hgr\XX_{\{a, b, c, d\}}^B\colon$} ;
\draw (1,1) node(11) {$\Hgr(\XX_{a}^{b})$} 
(3,1) node(31) {$\Hgr(V_{b})$}
(5,1) node(51) {$\Hgr(U^{c})$} ;
\draw (4,2) node(42) {$\Hgr(\XX_{b}^c)$} ;
\draw (2,2) node(22){$\Hgr(V_{a})$}
(-1,1) node(011){$0$}
(-1,3) node(013){$0$}
(1,3) node(13){$\Hgr(V_{a}, \XX_{a}^{a})$}
(0,4) node(04){$\Hgr(V_{a}, \XX_{a}^{b})$}
(0,2) node(02){$\Hgr(\XX_{a}^{b}, \XX_{a}^{a})$}; 
\draw (0,0) node(00){$\Hgr(\XX_{a}^{a})$} (2,0) node(20){$\Hgr(\XX_{b}^{b})$} (4,0) node(40){$\Hgr(Y)$} (6,0) node(60){$\Hgr(\XX_{c}^{c})$}  (6,2) node(62){$\Hgr(U^{d})$} (8,0) node(80){$\Hgr(\XX_{d}^{d})$} (7,1) node(71){$\Hgr(\XX_{c}^{d})$} (8,2) node(82){$\Hgr(\XX^d_c, \XX_{d}^{d})$}   (9,1) node(91){$0$}   (7,3) node(73){$\Hgr(U^d, \XX_{d}^{d})$}  (9,3) node(93){$0$} (8,4) node(84){$\Hgr(U^d, \XX_{c}^{d})$} ;
\draw[->] (31) -- (22); 
\draw[->] (31) -- (42); 
\draw[->] (11) -- (22); 
\draw[->] (00) -- (11); 
\draw[->] (20) -- (11);
\draw[->] (20) -- (31);
\draw[->] (40) -- (31);
\draw[->] (011) -- (02);
\draw[->] (11) -- (02);  
\draw[->] (00) -- (011);
\draw[->] (60) -- (71);  
\draw[->] (80) -- (71);
\draw[->] (60) -- (51);
\draw[->] (51) -- (42);
\draw[->] (02) -- (013);
\draw[->] (40) -- (51);

\draw[->] (02) -- (13);
\draw[->] (22) -- (13);

\draw[->] (013) -- (04);
\draw[->] (13) -- (04);

\draw[->] (51) -- (62);
\draw[->] (71) -- (62);
\draw[->] (71) -- (82);

\draw[->] (82) -- (73);
\draw[->] (62) -- (73);

\draw[->] (80) -- (91);
\draw[->] (91) -- (82);
\draw[->] (82) -- (93);
\draw[->] (93) -- (84);
\draw[->] (73) -- (84);
\end{tikzpicture}
\]
Additionally, we will need the following lemma:
\begin{lemma}\label{kernel}
Let $X$ be a compact subspace of a compact space $Z$, $Y$ a finite simplicial complex contained in $X$ and $X_i$ a countable nested family of compact spaces such that ${\cap_i X_i = X}$. Let $\Hgr$ be a \v{C}ech homology functor with coefficients in a field.
In diagrams
\[
\begin{array}{ccc}
\raisebox{-1.45ex}{\begin{tikzpicture}[xscale=2.3,yscale=1]
\draw (1.15,1) node(11){$\Hgr(X)$} (2,1) node(20){$\Hgr(X_i)$} ;
\draw[->] (11)  to node[above]{$j_i$} (20);
\end{tikzpicture}}
&
 \textrm{and}
 &
\raisebox{-1.45ex}{\begin{tikzpicture}[xscale=2.8,yscale=1]
\draw (0.4,1) node(00) {$\Hgr(Y)$};
\draw (1.15,1) node(11){$\Hgr(Z, X)$} (2,1) node(20){$\Hgr(Z, X_i)$} ;
\draw[->] (00) to node[above]{$q_Y$} (11); 
\draw[->] (11)  to node[above]{$q_i$} (20);
\end{tikzpicture}}  
\end{array}
\]
maps $j_i$, $q_Y$ and $q_i$ are induced by inclusions. The following equalities hold:
\[
\quad
\cap_{i} \Ker j_i = 0 \quad \textrm{and}\quad \Ker q_Y = \cap_{i} \Ker q_i \circ q_Y.
\]
\end{lemma}
\begin{proof}
By continuity of \v{C}ech homology~\cite{foundations}
\[
\varprojlim \Hgr(Z, X_i) = \Hgr \varprojlim (Z, X_i) = \Hgr (Z, X).
\]
The map
\[
\begin{tikzpicture}[xscale=1.8,yscale=1]
\draw  (2,1) node(21) {$\id_{\Hgr (Z, X)} \colon \varprojlim \Hgr(Z, X_i)$} (3.95,1) node(31){$\Hgr(Z, X)$};
\draw[->] (21) -- (31); 
\end{tikzpicture}
\]
satisfies the compatibility conditions for inverse limits and by the universal property equals $\varprojlim q_i$. Similarly, $\varprojlim j_i = \id_{\Hgr(X)}$.

Since the inverse limit functor preserves kernels,
\[
\varprojlim \Ker j_i = \Ker \varprojlim  j_i = \Ker \id_{\Hgr(X)} = 0
\]
and
\[
\varprojlim \Ker q_i\circ q_Y =  \Ker \varprojlim (q_i\circ q_Y) =   \Ker \varprojlim q_i \circ \varprojlim q_Y =   \Ker \id_{\Hgr (Z, X)} \circ q_Y = \Ker q_Y.
\]
The statement follows since the inverse limit of a nested sequence of vector spaces is precisely their intersection. An identical argument proves the second statement.
\end{proof}

\begin{proof}[Proof of Theorem~\ref{frosinithm}]
Let $(p, q) \in \RR^2$ be such that $p < q < \infty$.

First we show that $(p^+, q^*)$ appears with multiplicity 0 in $\Dgm^\dus(\XX)$ and $\Dgm^\dds(\XX)$. It suffices to prove that
\[
\lim_{\epsilon \to 0} \mudu_{\XX}([p, p+\epsilon] \times [c, d]) = 0\quad \textrm{and}\quad \lim_{\epsilon \to 0} \mudd_{\XX}([p, p+\epsilon] \times [c, d]) = 0.
\]

Let $m$ and a descending sequence of positive numbers $\epsilon_1 \geq \epsilon_2 \geq \ldots \geq 0$ be such that ${\displaystyle \lim_{i \to \infty} \epsilon_i = 0}$ and $p+3\epsilon_1 < m < c-3\epsilon_1$. Then
 \[
 \mudu_{\XX}([p, p+\epsilon] \times [c, d]) = \left\langle \bat{1}\,|\, \Hgr\XX_{\{p, p+\epsilon, c, d\}}^B \right\rangle
 \]
 and
  \[
  \mudd_{\XX}([p, p+\epsilon] \times [c, d])=\left\langle \bat{14}\,|\, \Hgr\XX_{\{p, p+\epsilon, c, d\}}^B \right\rangle.
\]

Using the Mayer--Vietoris and the restriction principles, we bound $ \mudu_{\XX}([p, p+\epsilon] \times [c, d])$ and $ \mudd_{\XX}([p, p+\epsilon] \times [c, d])$:
\[
\begin{array}{lclcl}
\left\langle \bat{1} \right\rangle &=& \left\langle \bat{5} \right\rangle & \leq &\left\langle \bat{6} \right\rangle  \\
&\leq &\left\langle \bat{8} \right\rangle &= & \left\langle \bat{10} \right\rangle \\
& = & \left\langle \bat{11} \right\rangle .&  &\\
\end{array}
\]
Similarly, 
\[
\left\langle \bat{14 } \right\rangle \leq \left\langle \bat{11} \right\rangle.
\]
By the restriction principle
\[
\begin{array}{lcl}
\dim \Ker \Hgr(Y \to (V_p, \XX_p^{p+\epsilon_i})) &=&  \left\langle \bat{12} \right\rangle \\
&=& \left\langle \bat{11} \right\rangle+ \left\langle \bat{12} \right\rangle\\
&=& \left\langle \bat{11} \right\rangle + \dim \Ker \Hgr(Y \to (V_p, \XX_p^p)). \\
\end{array}
\]
By Lemma~\ref{kernel}
\[
\cap_i \Ker \Hgr(Y \to (V_p, \XX_p^{p+\epsilon_i})) =  \Ker \Hgr(Y \to (V_p, \XX_p^p)).
\]
Since $\Ker \Hgr(Y \to (V_p, \XX_p^{p+\epsilon_i}))$ and $\Ker \Hgr(Y \to (V_p, \XX_p^p))$ are all finite dimensional, 
\[
\dim \Ker \Hgr(Y \to (V_p, \XX_p^p)) = \lim_{i\to \infty} \dim \Ker \Hgr(Y \to (V_p, \XX_p^{p+\epsilon_i})).
\]
This implies that 
\[
\lim_{i\to \infty} \left\langle \bat{11} \right\rangle = 0.
\]
For all $i$
\[
0 \leq \mudu_{\XX}([p, p+\epsilon_i] \times [c, d]), \mudd_{\XX}([p, p+\epsilon_i] \times [c, d]) \leq \left\langle \bat{11} \right\rangle.
\]
As we let $i \to \infty$, the desired statement follows.

By symmetry $(p^*, q^-)$ appears with multiplicity 0 in $\Dgm^\dus(\XX)$ and $\Dgm^\uus(\XX)$.

Next we prove that $(p^*, q^+)$ appears with multiplicity 0 in $\Dgm^\dds(\XX)$ and $\Dgm^\uds(\XX)$, ie.
\[
\lim_{\epsilon \to 0} \mudd_{\XX}([a, b] \times [q, q+\epsilon]) = 0 \quad \textrm{and}\quad \lim_{\epsilon \to 0} \muud_{\XX}([a, b] \times [q, q+\epsilon]) = 0.
\]
Let $m$ and a descending sequence of positive numbers $\epsilon_1 \geq \epsilon_2 \geq \ldots \geq 0$ be such that ${\displaystyle \lim_{i \to \infty} \epsilon_i = 0}$ and $b+3\epsilon_1 < m < q-3\epsilon_1$. %The number $\mudd_{\XX}([a, b] \times [q, q+\epsilon])$ counts the occurrences of $\bat{4}$ in the interval decomposition of $\Hgr\XX_{\{a, b, q, q+\epsilon\}}^B$.
 Since all the diamonds are Mayer--Vietoris
\[
\left\langle \bat{4} \right\rangle =  \left\langle \bat{2} \right\rangle \leq  \left\langle \bat{3} \right\rangle.
\]
Note that 
\[
 \left\langle \bat{3} \right\rangle
 =
 \dim
 \left[
 \Ker \Hgr(U^q \to U^{q+\epsilon_i}) \cap \Img \Hgr(Y \to U^{q}).
 \right]
 \]
Vector spaces $\Ker \Hgr(U^q \to U^{q+\epsilon_i})  \cap \Img \Hgr(Y \to U^{q})$ are finite dimensional subspaces of $\Ker \Hgr(U^q \to U^{q+\epsilon_i})$ ($Y$ is a finite simplicial complex and therefore has finitely generated homology groups).
By Lemma~\ref{kernel} (it applies since Steenrod--Sitnikov and \v{C}ech homology coincide for a certain choice of coefficients)
\[
\cap_{i} \Ker \Hgr(U^q \to U^{q+\epsilon_i}) = 0.
\]
Consequently,
\[
\lim_{i\to\infty} \left\langle \bat{3} \right\rangle = \lim_{i\to\infty} \dim \Ker \Hgr(U^q \to U^{q+\epsilon_i})  \cap \Img \Hgr(Y \to U^{q}) = 0.
\]
Since
\[
0 \leq  \mudd_{\XX}([a, b] \times [q, q+\epsilon_i]) \leq \left\langle \bat{3} \right\rangle,
\]
${\displaystyle \lim_{i\to\infty} \mudd_{\XX}([a, b] \times [q, q+\epsilon_i]) = 0}$ and consequently $(p^-, q^*)$ appears with multiplicity 0 in the diagram determined by $\mudd_{\XX}$. If we bound $\mudd_{\XX}([a, b] \times [q, q+\epsilon])$ by the same term, we also get $\displaystyle \lim_{\epsilon \to 0} \muud_{\XX}([a, b] \times [q, q+\epsilon]) = 0$.

By symmetry $(p^-, q^*)$ appears with multiplicity 0 in $\Dgm^\uus(\XX)$ and $\Dgm^\uds(\XX)$. The statement follows.

%Next we show that $(p^*, q^+)$ appears with multiplicity 0 in $\Dgm^\dds(\XX)$. It suffices to prove that %limit over rectangles of type $[a, b] \times [q, q+\epsilon]$, where $a < p < b$ is 0 (see Subsection \ref{recmeasures}):
%\[
%\lim_{i\to\infty} \mudd_{\XX}([p, p+\epsilon_i] \times [c, d]) = 0
%\]
%for $p < m < c$ and a such descending sequence $\epsilon_1 \geq \epsilon_2 \geq \ldots \geq 0$ with ${\displaystyle \lim_{i \to \infty} \epsilon_i = 0}$ that $p+2\epsilon_1 < m < c - 2\epsilon_1$. 

%The number $\mudd_{\XX}([p, p+\epsilon_i] \times [c, d])$ counts the occurrences of $\bat{1}$ in the interval decomposition of $\Hgr\XX_{\{p, p+\epsilon_i, m, c, d\}}$.
 
%  The proof for other measures is similar. We use the diamond principle to bound the appropriate zigzag module multiplicity with a multiplicity of a zigzag module supported over $\vee$ in the batman diagram and that goes to 0 as we let $\epsilon_i$ go to 0. 
\end{proof}
\begin{remark}
The statement of Theorem~\ref{frosinithm} can be strengthened to include $\RR$-spaces $(X, f)$, where:
\begin{itemize}
\item
$X$ is a Euclidean neighborhood retract and $f$ is a proper continuous map (see~\cite{Landi_2010}). This works because such an $f$ can be approximated with a continuous $g$ whose slices and levelsets are retracts of finite simplicial complexes and therefore have finitely generated homology groups.
\item
$X$ is a compact ANR and $f$ is a continuous function (see~\cite{Burghelea_Haller_2013, Burghelea_Dey_2013, Burghelea_2015}). Any $f$ can be approximated by a continuous map $g$ whose slices and levelsets are compact ANR. Compact ANR's have finitely generated homology groups~\cite{compactANR}.
\end{itemize}
 \end{remark}

%-----------------------------------------------------------------
\subsection{Stability}
\label{subsec:stability}

Given an $\RR$-space $\XX = (X,f)$ with a well-defined parametrized homology, what is the effect on the persistence diagrams of a small perturbation of the function? Will the resulting diagram be `close' to the original? We can measure this in terms of the \emph{bottleneck distance}, a standard and widely used metric on persistence diagrams~\cite{Cohen-Steiner_2007}.

The bottleneck distance compares undecorated diagrams. Let $A,B$ be locally finite multisets defined in open sets $\mathcal{F}_A, \mathcal{F}_B$ in the extended plane $\overline\RR^2$.
Consider a partial bijection $\approx$ between $A$ and $B$. The `cost' of a partial bijection is defined
\[
\cost(\approx) = \sup \left\{
\begin{array}{ll}
\metric^\infty ((p,q),(r,s)) 
	& \textrm{matched pairs } (p,q) \approx (r, s)
\\
\metric^\infty((p, q), \overline{\RR}^2 - \mathcal{F}_B)
	& \textrm{if }(p, q) \in A \textrm{ is unmatched}
\\
\metric^\infty((r, s), \overline{\RR}^2 - \mathcal{F}_A)
	& \textrm{if }(r, s) \in B\textrm{ is unmatched}
\end{array}\right.
\]
and the bottleneck distance is then
\[
\bottle({A}, {B}) =
\inf
\left\{
\cost(\approx)
\mid
\text{$\approx$ is a partial bijection between $A$ and~$B$}
\right\}
\]
One can show using a compactness argument that the infimum is attained{\cite[Theorem 5.12]{Thestructureandstability}}.
In the definition we are using the $l^\infty$-metric in the extended plane,
\[
\metric^\infty ((p, q), (r, s)) = \max \{|p-r|, |q-s| \}
\]
with $|(+\infty) - (+\infty)| = |(-\infty) - (-\infty)| = 0$.
The distance to a subset is defined in the usual way. Note that the distance to $\overline{\RR}^2 - \Hp$ is equal to the distance to the diagonal, that being the more familiar formulation.

We reach our stability theorem for parametrized homology (Theorem~\ref{thm:stability}) by using a stability theorem from~\cite{Thestructureandstability} for diagrams of r-measures. There is a natural way to compare two r-measures. For $R = [a,b] \times [c,d]$ define the $\delta$-thickening $R^\delta = [a-\delta, b+\delta] \times [c-\delta, d+\delta]$. (For infinite rectangles, we use $-\infty-\delta = -\infty$ and $+ \infty + \delta = + \infty$.)
We say that two r-measures satisfy the \emph{box inequalities with parameter~$\delta$} if
\[
\mu(R) \leq \nu(R^\delta),
\quad
\nu(R) \leq \mu(R^\delta)
\]
for all $R$. Either inequality is deemed to be vacuously satisfied if $R^\delta$ exceeds the finite support of the measure on the right-hand side.

It is natural to hope that two measures $\mu, \nu$ which satisfy the box inequalities with parameter~$\delta$ will determine diagrams with bottleneck distance bounded by~$\delta$. This is unfortunately not true, and in fact there is no universal bound on the bottleneck distance between the two diagrams. However, with stronger assumptions, namely the existence of a 1-parameter family interpolating between $\mu$ and $\nu$, such a statement holds.

\begin{theorem}
[Stability for finite measures~{\cite[Theorem 5.29]{Thestructureandstability}}]
\label{stabilityfin}
Suppose $(\mu_t\, |\, t \in [0, \delta])$ is a 1-parameter family of finite r-measures on $\Hp$. Suppose for all $s, t\in [0, \delta]$ the box inequality
\[
\mu_s(R) \leq \mu_t(R^{|s-t|})
\]
holds for all~$R$. Then there exists a $\delta$-matching between $\dgm(\mu_0)$ and  $\dgm(\mu_\delta)$.
\qed
\end{theorem}

We now apply this to the situation at hand.

\begin{lemma}[Box lemma]
\label{lem:box}
Let $\XX = (X, f)$, $\mathbb{Y} = (X, g)$ be $\RR$-spaces with $\Hgr$-taut fibers on the same total space~$X$.
Write $\mu^\xxs = \mu^\xxs_{\Hgr\XX}$ and $\nu^\xxs = \mu^\xxs_{\Hgr\YY}$ for $\xxt = \udt, \ddt, \uut, \dut$. Then
\[
\mu^\xxs (R)\leq \nu^\xxs (R^{\delta})
\quad \textrm{and}\quad
\nu^\xxs (R)\leq \mu^\xxs (R^{\delta})
\]
for any $\delta > \| f - g \|$.
\end{lemma}

\begin{proof}
We only need to consider rectangles $R = [a,b] \times [c,d]$ whose $\delta$-thickening is contained in~$\Hp$. This implies, in addition to ${a < b < c < d}$, that ${b + \delta < c - \delta}$.

The proof requires four different kinds of interlevelset. When $p \leq q$ we have the familiar
\begin{align*}
\XX_p^q = \{ x \in X \mid p \leq f(x) \leq q \},
\qquad
\YY_p^q = \{ x \in X \mid p \leq g(x) \leq q \},
\end{align*}
and when $p + \delta \leq q$ we define two new kinds,
\begin{align*}
\UU_p^q = \{ \text{$p \leq f(x)$ and $g(x) \leq q$}\},
\quad
\VV_p^q =\{ \text{$p \leq g(x)$ and $f(x) \leq q$} \}.
\end{align*}
In other words, $\UU_p^q$ is the space cut out between $f^{-1}(p)$ on the left and $g^{-1}(q)$ on the right. The condition $p + \delta \leq q$ ensures that $\UU_p^q$ and $\VV_p^q$ separate $\XX$ in the obvious way:
\begin{alignat*}{4}
\XX &= \XX^p \cup \UU_p^q \cup \YY_q
&&\quad\text{with}\quad
\XX^p \cap \UU_{p}^{q} = \XX_{p}^{p}
&&\quad\text{and}\quad
&\UU_{p}^{q} \cap \YY_{q} &= \YY_{q}^{q}
\\
\XX &= \YY^p \cup \VV_p^q \cup \XX_q
&&\quad\text{with}\quad
\YY^p \cap \VV_{p}^{q} = \YY_{p}^{p}
&&\quad\text{and}\quad
&\VV_{p}^{q} \cap \XX_{q} &= \XX_{q}^{q}.
\end{alignat*}

Consider the following Himalayan diagram:
\[
\begin{tikzpicture}[xscale=1.1,yscale=1.1]
\draw (1,1) node(11) {$\VV_{a-\delta}^{a}$} 
(3,1) node(31) {$\XX_{a}^b$}
(5,1) node(51) {$\UU_{b}^{b+\delta}$} ;
\draw (4,2) node(42) {$\UU_{a}^{b+\delta}$} ;
\draw (6,2) node(62) {$\UU_{b}^{c-\delta}$} ;
\draw (7,3) node(73) {$\XX_{b}^{c}$} ;
\draw (8,2) node(82) {$\VV_{b+\delta}^{c}$} ;

\draw (9,1) node(91) {$\VV_{c-\delta}^{c}$} ;
\draw (10,2) node(102) {$\VV_{c-\delta}^{d}$} ;
\draw (11,3) node(113) {$\YY_{c-\delta}^{d+\delta}$} ;
\draw (11,1) node(111) {$\XX_{c}^{d}$} ;
\draw (12,2) node(122) {$\UU_{c}^{d+\delta}$} ;
\draw (13,1) node(131) {$\UU_{d}^{d+\delta}$} ;

\draw (2,2) node(22){$\VV_{a-\delta}^b$};
\draw (3,3) node(33){$\YY_{a-\delta}^{b+\delta}$};
\draw (0,0) node(00){$\YY_{a-\delta}^{a-\delta}$} (2,0) node(20){$\XX_{a}^{a}$} (4,0) node(40){$\XX_b^b$} (6,0) node(60){$\YY_{b+\delta}^{b+\delta}$} (8,0) node(80){$\YY_{c-\delta}^{c-\delta}$} (7,1) node(71){$\YY_{b+\delta}^{c-\delta}$}   (10,0) node(100){$\XX_{c}^{c}$}  (12,0) node(120){$\XX_{d}^{d}$}  (14,0) node(140){$\YY_{d+\delta}^{d+\delta}$}  ;

\draw[->] (22) -- (33); 
\draw[->] (20) -- (31); 
\draw[->] (42) -- (33); 
\draw[->] (51) -- (62); 
\draw[->] (62) -- (73);
\draw[->] (71) -- (62); 
\draw[->] (71) -- (82); 
\draw[->] (82) -- (73); 
\draw[->] (91) -- (82); 
\draw[->] (91) -- (102); 
\draw[->] (80) -- (91); 
\draw[->] (102) -- (113); 
\draw[->] (100) -- (91); 
\draw[->] (100) -- (111); 
\draw[->] (111) -- (102); 
\draw[->] (120) -- (111); 
\draw[->] (120) -- (131); 
\draw[->] (131) -- (122);
\draw[->] (131) -- (140);  

\draw[->] (111) -- (122); 

\draw[->] (122) -- (113); 

\draw[->] (31) -- (22); 
\draw[->] (31) -- (42); 
\draw[->] (11) -- (22); 
\draw[->] (00) -- (11); 
\draw[->] (20) -- (11); 

\draw[->] (40) -- (31);

\draw[->] (60) -- (71);  
\draw[->] (80) -- (71);
\draw[->] (60) -- (51);
\draw[->] (51) -- (42);
\draw[->] (40) -- (51);

\end{tikzpicture}
\]
The nine diamonds of this diagram are Mayer--Vietoris. This is automatic for the top three diamonds. For the lower six diamonds we use the $\Hgr$-tautness of the fibers of $\XX$ and $\YY$, and the fact that the space at the top of each diamond is a normal neighborhood of the fiber, since $\delta > \|f - g\|$.

Applying $\Hgr$ to the diagram, we calculate (for example):
\begin{alignat*}{2}
\nudu(R^\delta)
&= \zzboxuu{1}
\\
&= \zzboxuu{2} &&+ (\text{eight other terms})
\\
&= \zzboxuu{3} &&+ (\text{eight other terms})
\\
&= \zzboxuu{4} &&+ (\text{eight other terms})
\\
&= \zzboxuu{5} &&+ (\text{eight other terms})
\\
&= \mudu(R) &&+ (\text{eight other terms}).
\end{alignat*}

To explain the second line, note that there are nine different summand types which restrict to the summand type in the first line: three possible start points  ${(\VV_{a-\delta}^{a}, \VV_{a-\delta}^{b}, \YY_{a-\delta}^{b+\delta})}$ times three possible end points ${(\YY_{c-\delta}^{d+\delta}, \UU_{c}^{d+\delta}, \UU_{d}^{d+\delta})}$. We are interested in only one of the nine terms.

Since the eight other terms are nonnegative, it follows that ${\mudu(R) \leq \nudu(R^\delta)}$ for all relevant~$R$. By symmetry, ${\nudu(R) \leq \mudu(R^\delta)}$ also. The calculations for $\xxt = \udt, \ddt, \uut$ are similar.
\end{proof}

\bigskip

%{\color{WildStrawberry}
\begin{theorem}[Stability of Parametrized Homology]
\label{thm:stability}
Let $\XX = (X, f)$ and $\mathbb{Y} = (X, g)$ be $\RR$-spaces with the same total space $X$ that satisfy one of the following conditions:

\smallskip
(i) $X$ is a locally compact polyhedron, $f$ and $g$ are proper, and $\Hgr$ is Steenrod--Sitnikov homology.

\smallskip
(iv)
$X$ is a locally compact polyhedron, $f$ and $g$ are proper piecewise-linear maps.

\smallskip
(v)
$X \subseteq \RR^n \times \RR$ is a closed definable set in some o-minimal structure and $f$ is the projection onto the second factor.

The associated r-measures for $\XX, \mathbb{Y}$ are written with the letters $\mu, \nu$ respectively. Then 
\[
\bottle (\dgm^\xxs(\Hgr\XX), \dgm^\xxs(\Hgr\YY)) \leq \|f-g\|
\]
for each type $\xxt = \udt, \ddt, \uut, \dut$.
\end{theorem}
%}

\bigskip

\begin{proof}
For any $\delta > \|f - g\|$ we can define the interpolating family
\[
f_t = (1 - (t/\delta)) f + (t/\delta) g
\]
for $t \in [0,\delta]$.
%
%%
%\[
%f_t = \frac{(\| f-g \|-t)}{\| f-g \|}f + \frac{t}{\| f-g \|}g
%\]
%%
%\[
%f_t = (1 - \lambda_t) f + \lambda_t g,
%\qquad
%\text{where $\lambda_t = \frac{t}{\| f-g \|}$}
%\frac{(\| f-g \|-t)}{\| f-g \|}f + \frac{t}{\| f-g \|}g
%\]
%%
%for $t\in [0, \| f-g\|]$. 
%
Note that $f_0= f$ and $f_{\delta} = g$. Since $f$ is proper and $\|f - f_t\|$ is bounded for all $t\in[0, \delta]$, $f_t$ are proper.  So each $(X, f_t)$ in situations (i) and (iv) determines an r-measure $\mu^\xxs_t$.
For any $s,t \in [0,\delta]$ we have $\|f_s - f_t\| < |s-t|$ and therefore
\[
\mu^\xxs_s(R)\leq \mu^\xxs_t (R^{|s-t|})
\]
by Lemma~\ref{lem:box}.
Theorem~\ref{stabilityfin} implies that there exists an $\delta$-matching between
\[
\dgm(\mu^\xxs_0) = \dgm(\mu^\xxs) = \dgm^\xxs(\Hgr\XX)
\quad
\text{and}
\quad
\dgm(\mu^\xxs_{\|f-g\|}) = \dgm(\nu^\xxs) = \dgm^\xxs(\Hgr\YY).
\]
Since this is true for all $\delta > \| f - g \|$ the result follows.
\end{proof}

%-----------------------------------------------------------------
%-----------------------------------------------------------------
%{\color{LimeGreen}
%The standard blow-up and classic persistence?
%}

%{\color{CornflowerBlue}
%Reeb graphs? Weak cosheaves?
%}

%{\color{Lavender}
%Stronger stability connecting $\Dgm^\dus(\Hgr_k\XX)$ with $\Dgm^\uds(\Hgr_{k-1}\XX)$?
%}

%{\color{Orange}
%Persistence and extended persistence
%}

%{\color{Gray}
%Functoriality?}

%\bigskip
%{\color{Red} Maybe the next bit belongs in the discussion at the end?}

%{\color{LimeGreen}
%\begin{example}\label{blowup}(The standard blow-up)
%
%{\color{Red} What is the standard blow-up for?}
%
%Let $(X_a \, |\, a\in \RR)$ be a 1-parameter family of subspaces of a fixed topological space $X_\infty$. We can represent this family as a parametrized space $\XX = (X, f)$, with total space
%\[
%X = \{(x, a)\, |\, x\in X_a, a\in \RR\},
%\]
%equipped with subspace topology inherited from  $X_\infty \times \RR$, and $f$ being the projection onto the second factor.
%\end{example}
%}    

%-----------------------------------------------------------------
\subsection{Extended persistence}
\label{sec:extended-persistence}

Closely related to ours is the work on \emph{extended persistence} by
Cohen-Steiner, Edelsbrunner, and Harer~\cite{ExtendingPersistence}.
Among other contributions, they construct four types of diagrams associated with an $\RR$-space. These diagrams can describe the geometry and
topology of a three-dimensional shape, a feature that finds applications in
protein docking~\cite{elevation}. In this section we explain how their four diagrams  correspond exactly with the four parametrized homology measures we have developed in this paper. 

Given an $\RR$-space $\XX = (X,f)$ they examine a concatenation of two sequences of spaces: a filtration of the sublevelsets of~$f$ and a filtration of pairs of
the space relative to the superlevelsets of~$f$.
\[
    \XX^{a_1}           \to
    \XX^{a_2}           \to
    \ldots              \to
    \XX^{a_n}           \to
    \XX = (\XX, \emptyset)  \to
    (\XX, \XX_{a_n})    \to
    \ldots              \to
    (\XX, \XX_{a_2})    \to
    (\XX, \XX_{a_1})
\]
The indices $a_1, a_2, \dots, a_n$ are taken to be the critical values of $f$; the underlying assumption of~\cite{ExtendingPersistence} being that we are in a Morse type situation. 

Within this sequence, four types of intervals are distinguished: those that
are supported on the absolute (ordinary) half of the sequence,
those supported on the relative half,
and those supported over both halves, in the latter further distinguishing intervals where the
superscript of the space associated to the left endpoint is lower or higher than
the subscript in the relative part of the right endpoint.

To translate their work into the language of measures,
for real numbers $a < b < c < d$ we consider a sequence of spaces:
\[
 \XX^{EP}_{a,b,c,d}: \XX^a        \to
                     \XX^b        \to
                     \XX^c        \to
                     \XX^d        \to
                     (\XX, \XX_d) \to
                     (\XX, \XX_c) \to
                     (\XX, \XX_b) \to
                     (\XX, \XX_a).
\]
We begin by translating their work into the language of measures. This, incidentally, removes the restrictive Morse-type hypothesis from the definition of the extended persistence diagram (see also [17] Section 6.2). For real numbers a < b < c < d we consider a sequence of spaces:
\begin{align*}
\muOrd_i ([a,b] \times [c,d])
	&= \langle \ordPic \mid \Hgr_i(\XX^{EP}_{a,b,c,d}) \rangle
\\
 \muRel_i ([a,b] \times [c,d])
 	&= \langle \relPic \mid \Hgr_i(\XX^{EP}_{a,b,c,d}) \rangle
\\
\muExtP_i ([a,b] \times [c,d])
    	&= \langle \extPPic \mid \Hgr_i(\XX^{EP}_{a,b,c,d}) \rangle
\\
\muExtM_i ([a,b] \times [c,d])
	&= \langle \extMPic \mid \Hgr_i(\XX^{EP}_{a,b,c,d}) \rangle.
\end{align*}
In the case of a Morse type $\RR$-space, we can retrieve the extended persistence intervals by restricting $a,b,c,d$ to the critical values $a_i$ of~$f$. However, these four measures are defined without that assumption.

The main result of this section expresses the relationship between the extended
persistence and the parametrized homology of the pair $\XX = (X, f)$.
Specifically, the four extended persistence measures are in one-to-one correspondence with the four
parametrized homology measures.

\begin{theorem}
Let $\Hgr$ be a homology functor with field coefficients and $\XX$ an $\RR$-space with $\Hgr$-taut levelsets. Then:
\begin{alignat*}{2}
\mudd_i &= \muOrd_i    \qquad\qquad
&
\muuu_i &= \muRel_{i+1}
\\
\mudu_i &= \muExtP_i 
&
\muud_i &= \muExtM_{i+1}
\end{alignat*}
Here we have abbreviated $\mu^\xxs_{\Hgr_i\XX}$ to $\mu^\xxs_i$ for each type $\xxt = \ddt, \dut, \uut, \udt$.
\end{theorem}

\begin{proof}
We prove the third equality; the rest are proven similarly.

\begin{figure}[h!]
    \centering
    \includegraphics[scale=1.25]{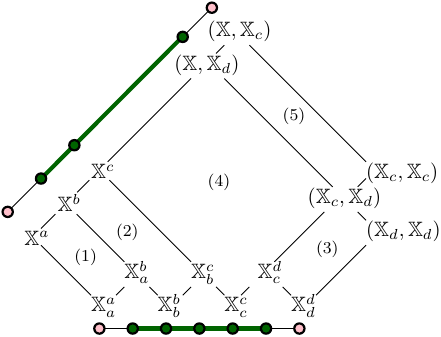}
    \caption{Diamonds involved in the proof $\mudu_i = \muExtP_i$.}
    \label{fig:lzz-ep-transformation}
\end{figure}

Repeatedly applying the Diamond Principle to the spaces in
Figure~\ref{fig:lzz-ep-transformation}, we get
\begin{align*}
  \mudu_i ([a,b] \times [c,d]) &= \muLzzEP{1}
              = \muLzzEP{2}
              = \muLzzEP{3} \\
             &= \muLzzEP{4}
              = \muLzzEP{5}
              = \muLzzEP{6} \\
             &= \muLzzEP{7} = \muExtP_i ([a,b] \times [c,d])
\end{align*}
for any rectangle $[a,b] \times [c,d]$. Thus the measures are equal.
\end{proof}

%-----------------------------------------------------------------
%-----------------------------------------------------------------
\section{Parametrized cohomology}
\label{sec:parcoho}
Let $\XX=(X, f)$ be a $\RR$-space, and let $\Hgr^*$ be a \emph{cohomology} functor with coefficients in a field~$\kk$. We define four persistence measures, and therefore four persistence diagrams, just as we did with homology functors.

\begin{remark}
The formalism applies equally well to extraordinary cohomology functors (over $\kk$).
\end{remark}

Here are the main steps.
For any rectangle $R = [a,b] \times [c,d]$, the zigzag diagram of spaces
\[
\XX_{\{a, b, c, d\}}:  \raisebox{-4.25ex}{
\begin{tikzpicture}[xscale=1.05,yscale=1.05]
\draw (1,1) node(11) {$\XX_a^b$} (3,1) node(31) {$\XX_b^c$} (5,1) node(51) {$\XX_c^d$} ;
\draw (0,0) node(00){$\XX_a^a$} (2,0) node(20){$\XX_b^b$} (4,0) node(40){$\XX_c^c$} (6,0) node(60){$\XX_d^d$}  ;
\draw[->] (00) -- (11); 
\draw[->] (20) -- (11);
\draw[->] (20) -- (31);
\draw[->] (40) -- (31);
\draw[->] (40) -- (51);
\draw[->] (60) -- (51);  
\end{tikzpicture}}
\]
becomes a zigzag diagram of vector spaces
\[
\Hgr^*\XX_{\{a, b, c, d\}}: 
\raisebox{-4.25ex}{
\begin{tikzpicture}[xscale=1.1,yscale=1.1]
\draw
(1,1) node(11) {$\Hgr^*(\XX_a^b)$}
(3,1) node(31) {$\Hgr^*(\XX_b^c)$}
(5,1) node(51) {$\Hgr^*(\XX_c^d)$} ;
\draw
(0,0) node(00){$\Hgr^*(\XX_a^a)$}
(2,0) node(20){$\Hgr^*(\XX_b^b)$}
(4,0) node(40){$\Hgr^*(\XX_c^c)$}
(6,0) node(60){$\Hgr^*(\XX_d^d)$.}  ;
\draw[->] (11) -- (00); 
\draw[->] (11) -- (20);
\draw[->] (31) -- (20);
\draw[->] (31) -- (40);
\draw[->] (51) -- (40);
\draw[->] (51) -- (60);  
\end{tikzpicture}}
\]
with the arrows reversed. Based on this diagram we define four measures
\begin{align*}
\muud_{\Hgr^*\XX}(R)
&= \langle\zzud \, |\, \Hgr^*\XX_{\{a, b, c, d\}} \rangle
\\
\mudd_{\Hgr^*\XX}(R)
&= \langle \zzdd \, |\, \Hgr^*\XX_{\{a, b, c, d\}} \rangle
\\
\muuu_{\Hgr^*\XX}(R)
&= \langle \zzuu \, |\, \Hgr^*\XX_{\{a, b, c, d\}} \rangle
\\
\mudu_{\Hgr^*\XX}(R)
&= \langle \zzdu \, |\, \Hgr^*\XX_{\{a, b, c, d\}} \rangle
\end{align*}
formally in the same way as before. The measures are additive if the fibers are $\Hgr^*$-taut (suitably defined), and finite if $\langle  \onof{6} \mid \Hgr^* \XX_{\{b, c\}}\rangle < \infty$. If both these conditions hold then four diagrams
\[
\Dgm^\uds(\Hgr^* \XX), \quad
\Dgm^\dds(\Hgr^* \XX), \quad
\Dgm^\uus(\Hgr^* \XX), \quad
\Dgm^\dus(\Hgr^* \XX)
\]
are defined. These diagrams constitute the parametrized cohomology of $\XX$.

\medskip
To a first approximation, there is no new information in parametrized cohomology.

\begin{theorem}
\label{prop:duality}
If $\Hgr^*$ is the cohomology functor dual to a homology functor~$\Hgr$, then the four diagrams for $\Hgr^* \XX$ are equal to the respective four diagrams for $\Hgr\XX$. 
\end{theorem}

\begin{proof}
The universal coefficient theorem gives a natural isomorphism of functors $\Hgr^*(-) \cong \Hom(\Hgr(-), \kk)$. This implies that there is an isomorphism of zigzag modules
\[
\Hgr^* \XX_{\{a,b,c,d\}} \cong \Hom( \Hgr\XX_{\{a,b,c,d\}}, \kk)
\]
for every $a < b \leq c < d$. So it is sufficient to prove that any zigzag module $\VV$ has the same interval-module multiplicities as its dual $\VV^* = \Hom(\VV, \kk)$.
More precisely, Proposition~\ref{prop:duality-2} will show that the finite multiplicities agree. This is enough, because the construction of a diagram from its measure does not discriminate between different infinite cardinalities.
\end{proof}

\begin{proposition}
\label{prop:duality-2}
Let $\VV$ be a zigzag module of length~$n$ and let $\VV^* = \Hom(\VV,\kk)$ be its dual.
Then, for all $1 \leq p \leq q \leq n$, we have
\[
\langle [p,q] \mid \VV \rangle = \langle [p,q] \mid \VV^* \rangle,
\]
with the understanding that all infinite cardinalities are regarded as equal.
\end{proposition}

Note that the shape of $\VV^*$ is the shape of $\VV$ with the arrows reversed, since $\Hom(-,\kk)$ is contravariant.
We write $\II[p,q]$ to denote the interval module supported over $[p,q]$ that has the same arrow orientations as~$\VV$. The corresponding interval module with opposite arrow orientations can be identified with its dual $\II[p,q]^* \cong \Hom(\II[p,q], \kk)$.

\begin{proof}
An interval decomposition of $\VV$ may be interpreted as an isomorphism
\[
\VV \cong \bigoplus_{p,q} V_{p,q} \otimes \II[p,q],
\]
where the direct sum ranges over $0 \leq p \leq q \leq n$, and where the $V_{p,q}$ are vector spaces. The interval multiplicities of~$\VV$ are given by the formula $\langle \II[p,q] \mid \VV \rangle = \dim(V_{p,q})$.
We take the dual of both sides to obtain 
\[
\VV^* \cong \bigoplus_{p,q} V_{p,q}^* \otimes \II[p,q]^*.
\]
This depends on two standard facts: (i) the dual of a finite direct sum of vector spaces is naturally isomorphic to the direct sum of the duals of the vector spaces; and (ii) the dual of the tensor product of a vector space and a finite-dimensional vector space is naturally isomorphic to the tensor product of the duals of the two vector spaces.
Thus
\[
\langle [p,q] \mid \VV \rangle
=
\langle \II[p,q] \mid \VV \rangle
=
\dim(V_{p,q})
\stackrel{\operatorname{fin}}{=}
\dim(V_{p,q}^*)
=
\langle \II[p,q]^* \mid \VV^* \rangle
=
 \langle [p,q] \mid \VV^* \rangle
\]
where $x \stackrel{\operatorname{fin}}{=} y$ means ``$x$ and $y$ are equal or are both infinite''.
\end{proof}

In practice, one may choose to describe a given diagram as parametrized homology or cohomology according to whichever seems more natural in the given context.
For example, here is a parametrized version of the classical Alexander duality theorem:

\begin{theorem}[Parametrized Alexander Duality~\cite{kalisnik, sarathesis}]\label{th:AD}
For $n\geq 2$, let $X \subset \RR^n \times \RR$, let ${Y = (\RR^n \times \RR) \setminus X}$,
and let $p: \RR^n \times \RR \to \RR$ be the projection onto the second factor. We assume that $(X,p)$ is proper, so that all levelsets $\XX_a^a$ and slices $\XX_a^b$ are compact. 
If parametrized \v{C}ech cohomology is defined for $\XX =(X, p|_X)$,
 then  it is also defined for $\YY=(Y, p|_Y)$. Additionally, for all $j  = 0, \ldots, n-1$:
%\[
%\begin{array}{lcl}
\begin{align*}
\Dgm^\uds(\widetilde{\Hgr}_{n-j-1}\YY)
&=
\Dgm^\dus(\check{\Hgr}^j \XX)
\\
\Dgm^\dds(\widetilde{\Hgr}_{n-j-1}\YY)
&=
\Dgm^\uus(\check{\Hgr}^j \XX)
\\
\Dgm^\uus(\widetilde{\Hgr}_{n-j-1}\YY)
&=
\Dgm^\dds(\check{\Hgr}^j \XX)
\\
\Dgm^\dus(\widetilde{\Hgr}_{n-j-1}\YY)
&=
\Dgm^\uds(\check{\Hgr}^j \XX)
\end{align*}
%\]
\end{theorem}

For the proof, we refer to~\cite{kalisnik, sarathesis}. Using this version of Alexander duality theorem, Henry Adams and Gunnar Carlsson~\cite{Adams01012015} provide a criterion for the existence of an evasion path in a sensor network.

%-----------------------------------------------------------------
%-----------------------------------------------------------------
\section*{Acknowledgements}

We thank Gregory Brumfiel and Matthew Wright for helpful discussions.

%-----------------------------------------------------------------
%-----------------------------------------------------------------
\bibliographystyle{plain}
\bibliography{bib}

\begin{thebibliography}{10}

\bibitem{Adams01012015}
H.~Adams and G.~Carlsson.
\newblock Evasion paths in mobile sensor networks.
\newblock {\em The International Journal of Robotics Research}, 34(1):90--104,
  2015.

\bibitem{elevation}
P.~K. Agarwal, H.~Edelsbrunner, J.~Harer, and Y.~Wang.
\newblock Extreme elevation on a 2-manifold.
\newblock {\em Discrete \& Computational Geometry}, 36:553--572, 2006.

\bibitem{Auslander}
M.~Auslander.
\newblock Representation theory of {A}rtin algebras {II}.
\newblock {\em Communications in Algebra}, 1:269--310, 1974.

\bibitem{Azumaya}
G.~Azumaya.
\newblock Corrections and supplementaries to my paper concerning
  {K}rull--{R}emak--{S}chmidt's theorem.
\newblock {\em Nagoya Mathematical Journal}, 1:117--124, 1950.

\bibitem{semialg}
S.~Basu, R.~Pollack, and M.F. Roy.
\newblock {\em Algorithms in Real Algebraic Geometry}.
\newblock Springer-Verlag, 2006.

\bibitem{matchings-barcodes}
Ulrich Bauer and Michael Lesnick.
\newblock Induced matchings of barcodes and the algebraic stability of
  persistence.
\newblock In {\em Proceedings of the Thirtieth Annual Symposium on
  Computational Geometry}, SOCG'14, pages 355:355--355:364, New York, NY, USA,
  2014. ACM.

\bibitem{bendich}
P.~Bendich, H.~Edelsbrunner, D.~Morozov, and A.K. Patel.
\newblock Homology and robustness of level and interlevel sets.
\newblock {\em Homology, Homotopy and Applications}, 15:51--72, 2013.

\bibitem{reflection}
J.H. Bernstein, I.M. Gelfand, and V.A. Ponomarev.
\newblock {C}oxeter functors and {G}abriel's theorem.
\newblock {\em Uspekhi Mat. Nauk}, 28:19--33, 1973.

\bibitem{Burghelea_2015}
D.~Burghelea.
\newblock A refinement of {B}etti numbers in the presence of a continuous
  function. {I}.
\newblock {\em Algebraic \& Geometric Topology}, 17(4):2051--2080, 2017.

\bibitem{Burghelea_Dey_2013}
D.~Burghelea and T.K. Dey.
\newblock Topological persistence for circle-valued maps.
\newblock {\em Discrete \& Computational Geometry}, 50(1):69--98, 2013.

\bibitem{Burghelea_Haller_2013}
D.~Burghelea and S.~Haller.
\newblock Topology of angle valued maps, bar codes and {J}ordan blocks.
\newblock {\em Journal of Applied and Computational Topology}, 1(1):121--197,
  Sep 2017.

\bibitem{Landi_2010}
F.~Cagliari and C.~Landi.
\newblock Finiteness of rank invariants of multidimensional persistent homology
  groups.
\newblock {\em Applied Mathematics Letters}, 24:516--518, 2011.

\bibitem{ZigzagPersistence}
G.~Carlsson and V.~de~Silva.
\newblock Zigzag persistence.
\newblock {\em Foundations of Computational Mathematics}, 10:367--405, 2010.

\bibitem{Zigzagpersistenthomologyandreal}
G.~Carlsson, V.~de~Silva, and D.~Morozov.
\newblock Zigzag persistent homology and real-valued functions.
\newblock {\em Proceedings of the 25th annual symposium on Computational
  geometry}, pages 247--256, 2009.

\bibitem{ZC}
G.~Carlsson and A.~Zomorodian.
\newblock Computing persistent homology.
\newblock {\em Discrete \& Computational Geometry}, 33:249--274, 2005.

\bibitem{CFFFL}
A.~Cerri, B.~Di~Fabio, M.~Ferri, P.~Frosini, and C.~Landi.
\newblock Betti numbers in multidimensional persistent homology are stable
  functions.
\newblock {\em Mathematical Methods in the Applied Sciences}, 36:1485--1648,
  2013.

\bibitem{Thestructureandstability}
F.~Chazal, V.~de~Silva, M.~Glisse, and S.~Oudot.
\newblock {\em The structure and stability of persistence modules}.
\newblock Springer International Publishing, 2016.

\bibitem{proximity}
Fr{\'e}d{\'e}ric Chazal, David Cohen-Steiner, Marc Glisse, Leonidas~J Guibas,
  and Steve~Y Oudot.
\newblock Proximity of persistence modules and their diagrams.
\newblock In {\em Proceedings of the twenty-fifth annual symposium on
  Computational geometry}, pages 237--246. ACM, June 2009.

\bibitem{Cohen-Steiner_2007}
D.~Cohen-Steiner, H.~Edelsbrunner, and J.~Harer.
\newblock Stability of persistence diagrams.
\newblock {\em Discrete {\&} Computational Geometry}, 37(1):103--120, 2007.

\bibitem{ExtendingPersistence}
D.~Cohen-Steiner, H.~Edelsbrunner, and J.~Harer.
\newblock Extending persistence using {P}oincar{\'e} and {L}efschetz duality.
\newblock {\em Foundations of Computational Mathematics}, pages 133--134, 2009.

\bibitem{pfd-persistence}
William Crawley-Boevey.
\newblock Decomposition of pointwise finite-dimensional persistence modules.
\newblock {\em Journal of Algebra and its Applications}, 14(05):1550066, August
  2014.

\bibitem{Silva_Munch_Patel_2015}
V.~de~Silva, E.~Munch, and A.~Patel.
\newblock Categorified reeb graphs.
\newblock {\em Discrete {\&} Computational Geometry}, 55(4):854--906, 2016.

\bibitem{StabilityCP}
T.K. Dey and R.~Wenger.
\newblock Stability of critical points with interval persistence.
\newblock {\em Discrete \& Computational Geometry}, 38:479--512, 2007.

\bibitem{pershom}
H.~Edelsbrunner, D.~Letscher, and A.~Zomorodian.
\newblock Topological persistence and simplification.
\newblock {\em Discrete \& Computational Geometry}, 28:511--533, 2002.

\bibitem{EMP}
H.~Edelsbrunner, D.~Morozov, and A.~Patel.
\newblock The stability of the apparent contour of an orientable 2-manifold.
\newblock {\em Topological Methods in Data Analysis and Visualization}, pages
  27--41, 2011.

\bibitem{foundations}
S.~Eilenberg and N.E. Steenrod.
\newblock {\em Foundations of algebraic topology}.
\newblock Princeton University Press, 1952.

\bibitem{Ferry}
S.C. Ferry.
\newblock Remarks on steenrod homology.
\newblock {\em Novikov Conjectures, index theorems and rigidity}, 2, 1995.

\bibitem{quiver}
P.~Gabriel.
\newblock {U}nzerlegbare {D}arstellungen {I}.
\newblock {\em Manuscripta Mathematica}, 6:71--103, 1972.

\bibitem{kalisnik}
S.~Kali\v{s}nik.
\newblock Alexander duality for parametrized homology.
\newblock {\em Homology, Homotopy and Applications}, 15:227--243, 2013.

\bibitem{sarathesis}
S.~Kali\v{s}nik.
\newblock {\em Persistent Homology and Duality}.
\newblock PhD thesis, University of Ljubljana, 2013.

\bibitem{Steenrod}
J.~Milnor.
\newblock On the {S}teenrod homology theory.
\newblock Mimeographed notes. Berkeley, 1960.

\bibitem{ominimal}
L.~van~den Dries.
\newblock {\em Tame Topology and O-minimal Structures}.
\newblock Cambridge University Press, 1998.

\bibitem{compactANR}
J.E. West.
\newblock Mapping {H}ilbert cube manifolds to {ANR}'s: A solution of a
  conjecture of {B}orsuk.
\newblock {\em Annals of Mathematics}, 106:1--18, 1977.

\end{thebibliography}

\bigskip
\bigskip

\noindent
\footnotesize {\bf Authors' addresses:}

\smallskip

\noindent Gunnar Carlsson,  \ Stanford University
\hfill {\tt gunnar@ayasdi.com}

\noindent Vin de Silva,
 \  Pomona College \hfill  {\tt vin.desilva@pomona.edu}

\noindent
Sara Kali\v snik,  \ Wesleyan University
\hfill {\tt skalisnikver@wesleyan.edu}

\noindent Dmitriy Morozov,
Lawrence Berkeley National Lab \hfill {\tt dmitriy@mrzv.org}

 \end{document}